\documentclass{article}
\usepackage[english]{babel}
\usepackage[T1]{fontenc}
\usepackage[utf8]{inputenc}
\usepackage{amsthm}
\usepackage{amsmath}
\usepackage{amssymb}
\usepackage{palatino}
\usepackage{graphicx}
\usepackage[colorlinks=true, citecolor =blue]{hyperref}
\usepackage{dsfont}
\usepackage[all]{xy}
\usepackage[top=2.5cm,bottom=2.5cm,left=2.5cm,right=2.5cm]{geometry}
\usepackage{enumitem}
\usepackage{mathtools}
\usepackage{subcaption}
\newcommand{\R}{\mathds{R}}
\newcommand{\N}{\mathds{N}}
\newcommand{\Z}{\mathds{Z}}

\newcommand{\C}{\mathcal{C}}

\theoremstyle{plain}

\newtheorem{thm}{Theorem}
\newtheorem{lem}{Lemma}
\newtheorem{prop}{Proposition}

\newtheorem{rmk}{Remark}

\renewenvironment{proof}{{\bfseries Proof.}}{\qed}
\title{An elapsed time model for strongly coupled inhibitory and excitatory neural networks}
\author{
Mar\'\i a J. C\'aceres\thanks{ Departamento de Matem\'atica Aplicada, Universidad de Granada,
E-18071 Granada, Spain.  Email:  caceresg@ugr.es }
\and
Beno\^\i t Perthame\thanks{Sorbonne Universit{\'e}, CNRS, Universit\'{e} de Paris, Inria, Laboratoire Jacques-Louis Lions UMR7598, F-75005 Paris, France. 
Email : Benoit.Perthame@sorbonne-universite.fr}
\and
Delphine Salort\thanks{Sorbonne Universit\'{e}, CNRS, Laboratoire de Biologie Computationnelle et Quantitative, UMR 7238, F-75005 Paris, France. Email: delphine.salort@sorbonne-universite.fr}
\and
Nicolas Torres\thanks{Sorbonne Universit{\'e}, CNRS, Universit\'{e} de Paris, Laboratoire Jacques-Louis Lions UMR7598, F-75005 Paris, France. 
Email : ???}
}
\date{March 2021}

\begin{document}

\maketitle

\begin{abstract}
    The elapsed time model has been widely studied in the context of mathematical neuroscience with many open questions left. The model consists of an age-structured equation that describes the dynamics of  interacting neurons structured by the elapsed time since their  last discharge. Our interest lies in  highly connected networks leading to strong nonlinearities where perturbation methods do not apply. To deal with this problem, we choose a particular case which can be reduced to delay equations.
    
    We prove a general convergence result to a stationary state in the inhibitory and the weakly excitatory cases. Moreover, we prove the existence of particular periodic solutions with jump discontinuities in the strongly excitatory case. Finally, we present some numerical simulations which ilustrate  various behaviors, which are consistent with the theoretical results.
\end{abstract}

\vskip .7cm

\noindent{\makebox[1in]\hrulefill}\newline
2010 \textit{Mathematics Subject Classification.}  82C32, 92B25, 35B10.
\newline\textit{Keywords and phrases.} Structured equations; Mathematical neuroscience; Neural networks; Periodic solutions; Delay differential equations.

\section{Introduction}
Understanding the dynamics of neural processes is  an interesting challenge from both mathematical and neuroscience viewpoint. Among the different models of neural assemblies, the elapsed time model has been widely studied  by several authors.  Long time convergence to steady state  is studied in  Pakdaman et al.  \cite{PPD, PPD3}, Cañizo et al.  \cite{canizo2019asymptotic}, Mischler et al. ~\cite{mischler2018}.  Periodic solutions for strongly excitatory cases are built in  Pakdaman et al. \cite{PPD2}. Modeling aspects  can be found in  Ly et al.~\cite{ly2009spike}, and Dumont et al. \cite{DH1, DH2}, Salort et al.~\cite{ST2020}.
These aspects include the relation with other neural models, convergence to steady states, synchronization phenomena and the existence of periodic solutions.
 For the derivation from microscopic models, see Chevallier et al.~\cite{chevallier2015microscopic} and Chevallier~\cite{chevallier2015mean}.

In the models we are interested in, neurons are subject to random discharges that interact with  the rest of the network and the population of neurons is described by their elapsed time since last discharge. More precisely, the dynamics are governed by the following nonlinear age-structured equation
\begin{equation}
\label{eq1}
\left\{
\begin{matrix*}[l]
\partial_t n+\partial_s n+p(s,N(t))n=0&t>0,s>0,\vspace{0.15cm}\\
N(t)= n(t,s=0)=\int_0^\infty p(s,N(t))n(t,s) \,ds&t>0,\vspace{0.15cm}\\
n(t=0,s)=n_0(s)\ge0&s\ge0,
\end{matrix*}
\right.
\end{equation}
where $n(t,\cdot)$ is the probability density of finding a neuron at time $t$, whose elapsed time since last discharge is $s\ge0$. The function $N(t)$ represents the flux of discharging neurons at time $t$, which corresponds to the activity of the network in this particular case. The function $p$ is the firing rate of neurons, which depends on the elapsed time $s$ and the activity $N$. We say that the network is inhibitory if $p$ is decreasing with respect to the activity
and excitatory if $p$ is increasing. Finally $n_0\in L^1_s$ is the initial data.

The case of strong nonlinearities has been investigated in  \cite{PPD2}, when the neurons only interact via the variation of the refractory period.  Here,  we  fix the refractory period and we assume that, after a fixed refractory state, the discharge rate of neurons follows an exponential law which parameter depends on the total activity via a smooth function $\varphi$. This leads to the following  particular form the firing rate $p$:
\begin{equation}
\label{pfire}
p(s,u)=\varphi(u)\mathds{1}_{\{s>\sigma\}},
\end{equation}
with  $\sigma>0$ the constant refractory period. We also assume that there exist two constants $p_0,\,p_\infty$ such that
\begin{equation}
\label{bdphi}
   0<  p_0\le\varphi(\cdot)\le p_\infty.
\end{equation}
For this particular form of the discharge rate $p$, the network is inhibitory when $\varphi'\le0$ and excitatory when $\varphi'>0$, 
in particular, is weakly excitatory if
$\varphi'$ is small.
Moreover, the System~\eqref{eq1} satisfies the mass conservation law, which reads
\begin{equation}
    \label{masscon}
    \int_0^\infty n(t,s)\,ds=\int_0^\infty n_0(s)\,ds=1,\quad\forall t\ge 0,
\end{equation}
and consequently we have the following bounds on $N$
\begin{equation*}
  0\leq N(t)\le p_\infty,\quad\forall t\ge0.
\end{equation*}
This specific form of $p$ allows us to investigate regimes with strong interactions as in \cite{PPD2}, reducing the problem to a delay equation (see \cite{diekmann2012delay,murray2007mathematical,smith2011introduction} for references).
Other standard methods, as Doeblin's theory in \cite{canizo2019asymptotic},   entropy method in $\cite{kang2015,PPD,PPD2,PPD3}$ or spectral methods in \cite{mischler2018}, provide results on exponential convergence to equilibrium but 
only for  weak nonlinearities.

For a general firing rate $p$ in System \eqref{eq1} it is conjectured that solutions converge to the unique steady state in the inhibitory case, whereas periodic solutions may arise in the excitatory case. This article is concerned to prove these conjectures for the specific form of the firing rate in \eqref{pfire}. 

The article is organized as follows. In section \ref{delay}, we show that System \eqref{eq1} can be studied through a delay differential equation. Moreover from a given periodic activity $N(t)$ solving the delay equation, we can recover a solution of System \eqref{eq1} by using the arguments in \cite{PPD2}. In section \ref{convergence}, we prove convergence to equilibrium in the inhibitory case and the weakly excitatory case. In addition we prove a  monotone convergence result in the excitatory case under certain conditions. Regarding periodic solutions, we prove in section \ref{periodic} the existence of piece-wise constant $\sigma$-periodic and with more elaborate arguments we also prove the existence of $2\sigma$-periodic solutions, which are piece-wise monotone. Finally, in section \ref{numerical} we show numerical examples with several possible behaviors like multiplicity of solutions, convergence to equilibrium in different ways and periodic solution with jump discontinuities.

\section{Reduction to a delay differential equation}
\label{delay}

For analysis purposes, we define the following function $\psi\colon[0,\infty)\to[0,\infty)$ as
\begin{equation}
\label{psi}
\displaystyle\psi(u)\coloneqq\frac{u}{\varphi(u)},
\qquad \textrm{ therefore }\qquad \psi'(u)=\frac{\varphi(u)-u\varphi'(u)}{\varphi^2(u)}.
\end{equation}
This function plays an important role in the study of System
\eqref{eq1}, because it can be reduced to a delay diffe\-ren\-tial
equation,
as one can see in the following lemma:
\begin{lem} We assume \eqref{pfire} and  \eqref{bdphi}.
For $t>\sigma$ the discharge flux $N(t)$ satisfies:
\begin{equation}
\label{mass}
    \int_{t-\sigma}^{t} N(s)\,ds+\psi(N(t))=1.
\end{equation}
Moreover, if $N(t)$ is smooth for $t>\sigma$, the following formula for $N'(t)$
holds:
\begin{equation}
\label{N'}
N(t)-N(t-\sigma)+\psi'(N(t))N'(t)=0.
\end{equation}
In the same way, for $0< t <\sigma$ the discharge flux satisfies
\begin{equation}
    \label{mass2}
    \int_0^t N(s)\,ds+\int_0^{\sigma-t} n_0(s)\,ds+\psi(N(t))=1,
\end{equation}
and if $N(t)$ is smooth we have the following formula for $N'(t)$:
\begin{equation}
    \label{N'2}
    N(t)-n_0(\sigma-t)+\psi'(N(t))N'(t)=0.
\end{equation}
\label{lm:delay}
\end{lem}

\begin{proof}
Using assumption \eqref{pfire}, the equation for $N(t)$ is rewritten as 
$$
N(t)=\varphi(N(t))\int_\sigma^\infty n(t,s)\,ds.
$$
And using the mass-conservation property and the method of characteristics
for $t > \sigma$,  we obtain:
$$
\int_\sigma^\infty n(t,s)\,ds =1- \int_0^\sigma  N(t-s)\,ds.
$$
Therefore, we get the first part of Lemma~\ref{lm:delay}. The result for $0<t<\sigma$ is proved in the same way.
\end{proof}

\

The sign of  $\psi'$ plays a crucial role in the behavior of the system
(see equation \eqref{N'}).
We will prove that complex dynamics can only occur when $\psi'$ changes sign. This does not happen for the inhibitory case, because
$
\psi'>0\quad\textrm{on}\quad [0,p_\infty],
$
and, in the  excitatory case if $\psi'>0$, on $[0,p_\infty]$.
In this latter case,  we say that the network is {\em weakly excitatory}.
For instance, this holds  if $0<\varphi'\le p_0/p_\infty$.
Otherwise, we say that the network is {\em strongly excitatory},
if $\psi'$ changes sign (and $\varphi'>0$).
\\

In the following theorem we show how to recover solutions of the original renewal model from a given activity that is a solution of
the integral equation~\eqref{mass}.

\begin{thm}[\bf Reconstructing a solution of \eqref {eq1} from a general activity]
  Assume \eqref{bdphi}. Let $N\in L^\infty(0,\infty)$ a non-negative function and satisfying  $\psi(N(t))\in \C([\sigma, \infty))\cap \C^1((0,\sigma))$ and the following conditions:
\begin{enumerate}
    \item \label{1st}$\widetilde{n}(t)\coloneqq N(\sigma-t)+\frac{d(\psi(N))}{dt}(\sigma-t)\ge0$ \quad  for \quad  $0<t<\sigma$,
    \item \label{2nd}$\int_{t-\sigma} ^t N(s)\,ds + \psi(N(t))=1$ \quad for   $t \geq \sigma$,  i.e.,  
      $\int_{0} ^\sigma N(s)\,ds + \psi(N(\sigma))=1$ and
      $N$ is solution of the integral equation \eqref{mass}.
\end{enumerate}
Then for any initial probability density $n_0$ satisfying
$n_0(s)=\widetilde{n}(s)$ for $0<s<\sigma$, the solution $n$ of the linear problem
\begin{equation}
\label{nlinear}
\left\{
\begin{matrix*}[l]
\partial_t n+\partial_s n+p(s,N(t))n=0&t>0, \, s>0,\vspace{0.15cm}\\
n(t,s=0)=N(t)&t>0,\vspace{0.15cm}\\
n(t=0,s)=n_0(s) &s>0,
\end{matrix*}
\right.
\end{equation}
determines a solution of Equation \eqref{eq1} with $N(t)$ as activity.
\label{th:reconstruction}
\end{thm}

From this theorem, we deduce that the behavior of $N(t)$ in Equation \eqref{eq1} is determined just by the initial data $n_0$ on $(0,\sigma)$ as long as it is a probability density.

\begin{proof}
In order to prove that the solution $n$ of Equation \eqref{nlinear} is actually a solution of Equation \eqref{eq1}, we must verify the following conditions for all $t>0$. 

\begin{equation}
\label{condsol}
\psi(N(t))=\int_\sigma^\infty n(t,s)\,ds
\quad
\mbox{and}
\quad
\int_0^\infty n(t,s)\,ds=1.
\end{equation}

Consider $M(t)\coloneqq\int_0^\infty n(t,s)\,ds$ and observe that
\begin{equation}
    \begin{split}
        M'(t)&=-\int_0^\infty\partial_sn(t,s)\,ds-\int_0^\infty p(s,N(t))n(t,s)\,ds\\
        &=N(t)-\varphi(N(t))\int_\sigma^\infty n(t,s)\,ds\\
        &=N(t)-\varphi(N(t))M(t)+\varphi(N(t))\int_0^\sigma n(t,s)\,ds.
    \end{split}
\end{equation}

For $0<t<\sigma$ we deduce from conditions \ref{1st} and \ref{2nd} that Equation \eqref{mass2} holds and by the method of characteristics we get that $M(t)$ satisfies that
\begin{equation}
    \begin{split}
        M'(t)&=N(t)-\varphi(N(t))M(t)+\varphi(N(t))\left(\int_0^t n(t,s)\,ds+\int_t^\sigma n(t,s)\,ds\right)\\
        &=N(t)-\varphi(N(t))M(t)+\varphi(N(t))\left(\int_0^t N(s)\,ds+\int_0^{\sigma-t} n_0(s)\,ds\right)\\
        &=N(t)-\varphi(N(t))M(t)+\varphi(N(t))\left(1-\psi(N(t))\right)\\
        &=\varphi(N(t))(1-M(t)).
    \end{split}
\end{equation}
Similarly for $t>\sigma$ we deduce from condition \ref{2nd} and the method of characteristics that for a.e. $t>\sigma$ the following equality holds
\begin{equation}
\begin{split}
    M'(t)&=N(t)-\varphi(N(t))M(t)+\varphi(N(t))\int_{t-\sigma}^t N(s)\,ds\\
         &=\varphi(N(t))\,(1-M(t)).
    \end{split}
\end{equation}

Therefore we obtain that $M(t)$ satisfies the following differential equation for a.e. $t>0$.  
\begin{equation*}
    \label{M'}
    \left\{
    \begin{matrix*}[l]
   M'(t)=\varphi(N(t))(1-M(t)),\vspace{0.15cm}\\
    M(0)=1.
    \end{matrix*}
    \right.
\end{equation*}
This implies that $M(t)\equiv1$ and we conclude \eqref{condsol} for all $t>0$, which implies that $n$ is a solution of~\eqref{eq1} with $N(t)$ as activity.
\end{proof}

In particular when the activity is periodic we can construct periodic solutions of \eqref{eq1}, using only the Lipschitz continuity of $\psi(N(t))$ induced by~\eqref{bdphi} and~\eqref{mass}.

\begin{thm}[\bf Reconstruction of a solution from a periodic activity]
\label{periodicn}
 Assume \eqref{bdphi}. For $T>0$, consider a non-negative $T$-periodic solution $N(t)$ of the integral equation \eqref{mass} for all $t\in\R$. Then, there exists a unique $T$-periodic solution $n$ of the linear Equation
\eqref{nlinear}, which corresponds to a solution of the nonlinear
System \eqref{eq1}.
\end{thm}

\begin{proof}
By periodicity we can extend the solution of Equation \eqref{nlinear} for all $t\in\R$ and from the method of characteristics we find a solution given by
$$n(t,s)=N(t-s)e^{-\int_0^s p(s',N(s'+t-s))\,ds'},$$
thus we can consider as initial data
$$n(t=0,s)=N(-s)e^{-\int_0^s p(s',N(s'-s))\,ds'},$$
which determines the unique $T$-periodic solution of \eqref{nlinear}. Therefore we can replicate the argument in the proof of Theorem \ref{th:reconstruction} in the case $t>\sigma$ to conclude the result.
\end{proof}


\section{Convergence to equilibrium}
\label{convergence}

By using the delay equation \eqref{N'},  we prove that solutions of System \eqref{eq1} converge to  a steady state in two cases.
In the inhibitory and weakly excitatory case, the  steady state is unique  and we prove global convergence. In the strongly excitatory case,  we prove a local convergence theorem with specific assumptions on $\psi$ and for the smooth solution.  For the function $p$ defined in \eqref{pfire}, these results extend the convergence beyond the standard  case of weak interconnections.

We recall here that in both the inhibitory and the weakly excitatory cases
$\psi'$ satisfies:
\begin{equation}
\label{psi'}
\psi'(u)=\frac{\varphi(u)-u\varphi'(u)}{\varphi^2(u)}>0,
\quad \mbox{for} \  u\in[0,p_\infty],
\end{equation}
while in the strongly excitatory case  $\psi'$  changes sign.
We know from \eqref{psi} and \eqref{mass} that there exists a unique steady state determined by $N^*$, which is the unique solution of the equation
\begin{equation}
\label{neq}
\sigma N^*+\psi(N^*)=1,
\ \mbox{and therefore,}
\quad
n^*=N^* e^{-\int_0^s p(s',N^*)\,ds'}.
\end{equation}
In the following theorem,
we prove the convergence to the steady state,
in the inhibitory and weakly excitatory cases.
\begin{thm}[\bf Convergence to equilibrium in the inhibitory or weakly excitatory cases]
  \label{convthm}
Assume that $\psi'$ satisfies \eqref{psi'}. Let $n$ be the solution of \eqref{eq1}. Then $n(t)\to n^*$ in $L^1_s$ and $N(t)\to N^*$ when $t\to\infty$, where $(n^*,N^*)$ corresponds to the unique steady state of equation \eqref{eq1} given by \eqref{neq}.
\end{thm}
In order to achieve the result, we need the following two lemmas.
\begin{lem}
Assume that $\psi$ satisfies \eqref{psi'}. Then the activity $N(t)$ can not be strictly monotone over any interval of length larger than $\sigma$. In particular there exists a sequence of local maxima $t_n$ (and resp. minima) such that $t_n\to\infty$.
\end{lem}
\begin{proof}
  Assume the existence an interval $I\subset(0,\infty)$, with $|I|>\sigma$, such that $N(t)$ is increasing on $I$. Let $t_0>\sigma$ such that $t_0,t_0-\sigma\in I$, so we have that $N(t_0-\sigma)<N(t_0)$. However we conclude from \eqref{N'} that $N'(t_0)<0$, which contradicts that $N(t)$ is increasing.
Similarly, if we assume that $N$ decreases on $I$ we obtain a contradiction. Therefore, we conclude that $N(t)$ increases and decreases many infinitely times. 
\end{proof}
\begin{lem}
\label{max}
If  $\psi$ satisfies \eqref{psi'}, then 
$$
\max_{t\in I_{k-1}} N(t)\ge\max_{t\in I_k}N(t),
\quad\left(\textrm{resp.}\,
\min_{t\in I_{k-1}}N(t)\le\min_{t\in I_k}N(t)\right),
$$
for $k\in\N$ and $I_k:=[k\sigma,(k+1)\sigma]$.
\end{lem}
\begin{proof}
We consider three cases to show that $\max_{t\in I_{k-1}} N(t)\ge\max_{t\in I_k}N(t)$:
\begin{enumerate}
\item If $\displaystyle\max_{t\in I_k}N(t)=N(k\sigma)$
  then the result is straightforward.
\item If $\displaystyle\max_{t\in I_k}N(t)=N((k+1)\sigma)$
  then $N'((k+1)\sigma)\ge0$ and from formula \eqref{N'} we get that $N(k\sigma)\ge N((k+1)\sigma)$ and the result is proved.
    \item If there exists $t_0\in\left(k\sigma,(k+1)\sigma\right)$ such that $\max_{I_k}N=N(t_0)$ then $N'(t_0)=0$ and by formula \eqref{N'} we have $N(t_0)=N(t_0-\sigma)$ and the result is proved.
\end{enumerate}
Analogously it is proved that
$\min_{t\in I_{k-1}}N(t)\le\min_{t\in I_k}N(t)$.
\end{proof}

Now we can prove Theorem \ref{convthm}.

\begin{proof}
  For $T>0$ we define the sequence $N_n(t)=N(t+n)$ for $t\in[-T,T]$ and $n\in\N$ such that $n-T>\sigma$. From formula \eqref{N'} we get that $N'$
  is uniformly bounded since $\psi'(u)>0$, for $u\in [0,p_\infty]$.
  This allows to conclude that after extraction of a sub-sequence, $N_n(t)$ converges uniformly in $[-T,T]$ to some function $N_\infty(t)$ defined on $\R$. Moreover $N_\infty(t)$ satisfies for all $t\in\R$ the equation \eqref{mass} so in particular $N_\infty\in\C^1_b(\R)$, and from Lemma~\ref{max}, we deduce that for all $k\in\Z$
  $$\max_{t\in I_k} N_\infty(t)=\limsup_{t\to\infty}N(t),
  \quad\left(\textrm{resp.}\, \min_{t\in I_k}N_\infty(t)=
  \liminf_{t\to\infty}N(t)\right),$$
  and in particular $N^+\coloneqq\max_{t\in \R}N_\infty(t)
  =\max_{t\in I_k} N_\infty(t)$ (resp. $N^-\coloneqq\min_{t\in \R}
  N_\infty(t)=\min_{t\in I_k} N_\infty(t)$).

  Next, we prove that $N_\infty(t)$ is a constant function.
  Choose $t^-, t^+\in\R$, with 
  $t^-<t^+$, such that $N(t^+)=N^+,\,N(t^-)=N^-$ and $|t^+-t^-|\le\sigma$. Since $N'(t^{\pm})=0$, from equation \eqref{N'} we get $N(t^{\pm})=N(t^{\pm}-\sigma)$ and thus by iteration we get
$$N(t^{\pm})=N(t^{\pm}-j\sigma),\quad\forall j\in\N.$$
Consider now $k\in\N$. By evaluating formula \eqref{mass} for $t=t^{\pm}-j\sigma$ for $j=1,\dots,k$ and taking the sum over $j$, the following equality holds
$$\int_{t^+-k\sigma}^{t^+} N_\infty(s)+k\psi(N^+)\,ds=\int_{t^--k\sigma}^{t^-} N_\infty(s)\,ds+k\psi(N^-),$$
and we conclude
$$\psi(N^+)-\psi(N^-)=\frac{1}{k}\left(\int_{t^-}^{t^+} N_\infty(s)\,ds-\int_{t^--k\sigma}^{t^+-k\sigma} N_\infty(s)\,ds\right).$$
Hence by taking $k\to\infty$ we conclude $\psi(N^+)=\psi(N^-)$ and this implies $N^+=N^-$, since $\psi$ is strictly increasing. Therefore $N_\infty$ is a constant function.

Since $N_\infty$ satisfies formula \eqref{mass}, we conclude by uniqueness of the solution of equation \eqref{neq} that $N_\infty=N^*$. In particular we conclude that $N(t)\to N^*$ when $t\to\infty$ and the convergence of $n(t)$ is obtained via the method of characteristics and Lebesgue's theorem.
\end{proof}

In strongly excitatory networks, $\psi'$ changes sign, therefore
 the above theorem does not apply. And also there is no
steady state uniqueness. In this case,
we prove the following local convergence theorem.
\begin{thm}[\bf Monotone convergence in the strongly excitatory case]
\label{monotoneconv}
Let $N^*>0$ be a solution of \eqref{neq} (i.e.  a steady-state
activity of \eqref{eq1}), and consider $N\in\C[0,\sigma]$ satisfying the following
conditions:
\begin{enumerate}
    \item $N(t)<N(\sigma)<N^*$ for all $t\in(0,\sigma)$ (resp. $N(t)>N(\sigma)>N^*$).
    \item $\int_0^\sigma N(s)\,ds+\psi(N(\sigma))=1$.
    \item $\psi'(u)<0$ \textrm{for} $u\in[N_-,N^*]$ with $N_-=\displaystyle\min_{t\in[0,\sigma]}N(t)$ (resp. for $u\in[N^*,N_+]$ with $N_+=\displaystyle\max_{t\in[0,\sigma]}N(t)$).
\end{enumerate}
Then there exists a strictly increasing (resp. decreasing) solution $N(t)$ of \eqref{mass}, which extends the given
$N\in\C[0,\sigma]$. Moreover if $N^*$ is the unique steady state activity lying on $[N(\sigma),N^*]$
(resp. on  $[N^*,N(\sigma)]$), then $N(t)\to N^*$ when $t\to\infty$.
\end{thm}
We point out that to be a solution of the time elapsed model \eqref{eq1},
we need to find a compatible initial data $n_0(s)$. 
Theorem \ref{th:reconstruction} can help in that direction.

\begin{proof}
We start with the proof of the case $N(t)$ increasing.
  Consider $\widetilde{\psi}\colon\R\to\R$ a smooth Lipschitz function such that for some constant $\alpha>0$ we have $\widetilde{\psi}'(u)<-\alpha$ and 
$$\widetilde{\psi}\equiv\psi\quad\textrm{on}\quad[N_-,N^*].$$
Thus the following delay differential equation
\begin{equation}
\label{delaypsitilde}
    \left\{\begin{matrix*}
    \frac{d}{dt}\widetilde{\psi}(u(t))=u(t-\sigma)-u(t)&t>\sigma,\\
    u(t)\equiv N(t)& t\in[0,\sigma],
    \end{matrix*}
    \right.
\end{equation}
has a unique global $u\in\C^1[\sigma^+,\infty)\cap\C[0,\infty)$ by applying Cauchy-Lipschitz theorem. We call this global solution as $N(t)$ as well. From equation \eqref{delaypsitilde} we observe that
$$\int_{t-\sigma}^{t} N(s)\,ds+\widetilde{\psi}(N(t))$$
is constant and from condition 2 we have that
\begin{equation}
\label{intpsitilde}
    \int_{t-\sigma}^{t} N(s)\,ds+\tilde{\psi}(N(t))=1.
\end{equation}
Next we prove that $N(t)<N^*$ for all $t>\sigma$. Suppose that there exists $t^*>\sigma$ such that $N(t^*)=N^*$ and $N(t)<N^*$ for $t<t^*$. Thus from equation \eqref{intpsitilde} we observe that
$$1=\int_{t^*-\sigma}^{t^*} N(s)\,ds+\psi(N^*)<\sigma N^*+\psi(N^*),$$
which contradicts equation \eqref{neq} and hence $N(t)<N^*$ for all $t>\sigma$.

Now we prove that $N'(t)>0$ for all $t>\sigma$. From equation \eqref{delaypsitilde} we have that $N'(\sigma^+)>0$ and from continuity there exists $\varepsilon>0$ such that $N'(t)>0$ for $t\in(\sigma,\sigma+\varepsilon)$. Suppose that there exists a first local maximum $t_*>\sigma$, so that $N'(t_*)=0$ and from equation \eqref{delaypsitilde} we get $N(t_*-\sigma)=N(t_*)>N(\sigma)$. If $t_*<2\sigma$ then we contradict condition 1 and if $t_*\ge2\sigma$ we contradict the monotony of $N(t)$ on $[\sigma,t_*]$. Therefore $N'(t)>0$ for all $t>\sigma$.

From monotony and boundedness of $N$, we conclude in particular that $N$ is a solution of equation \eqref{mass} for $t>\sigma$. Moreover, if $N^*$ is the unique steady state on the interval $[N(\sigma),N^*]$, then it is straightforward that $N(t)\to N^*$ when $t\to\infty$. The proof of $N(t)$ decreasing is analogous.
\end{proof}

Figure \ref{fig-monotoneconv} shows an  example of this theorem.
\begin{figure}[ht!]
    \centering
    \includegraphics[scale=0.5]{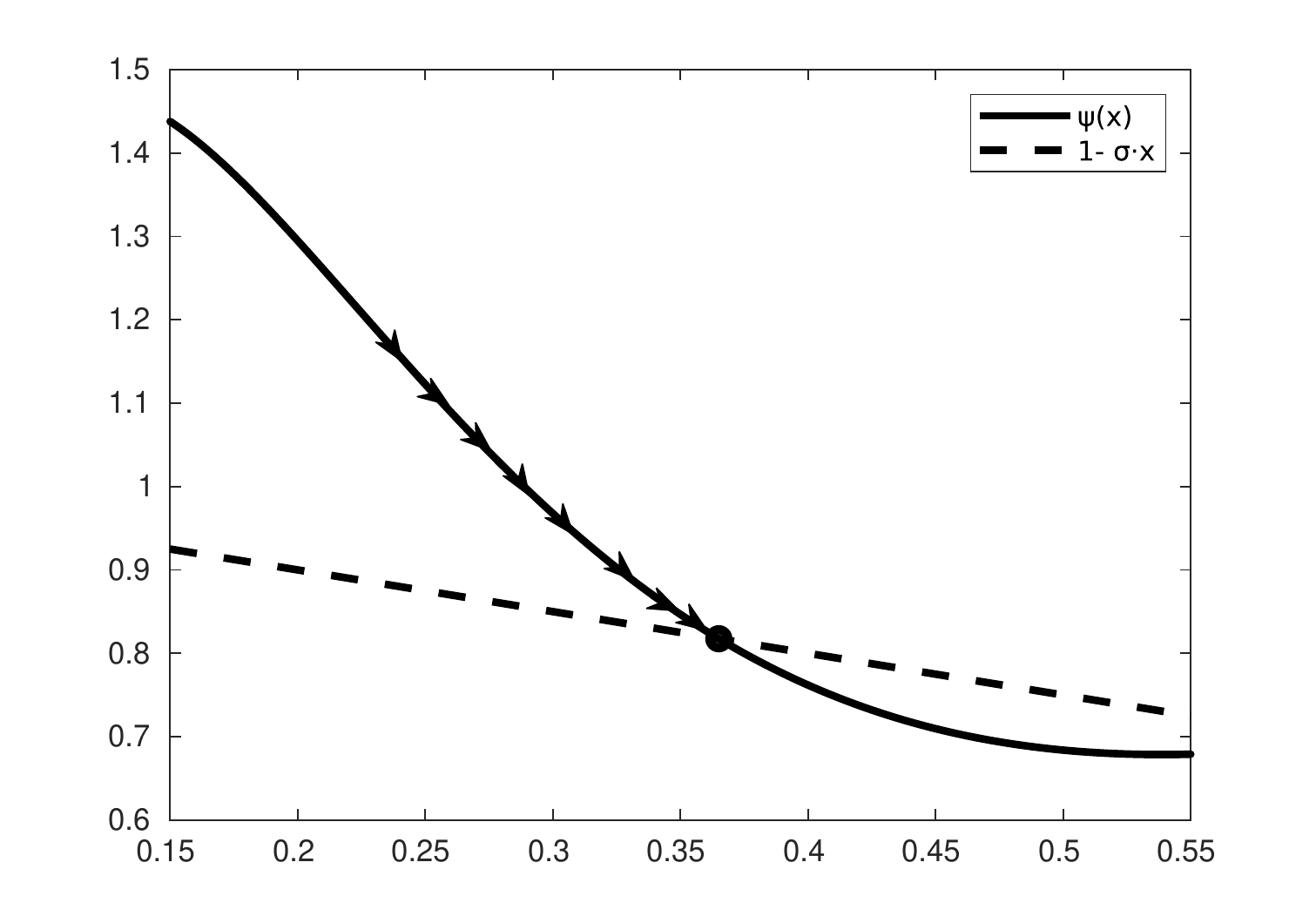}
      \\[-5pt]
    \caption{An example of an increasing solution converging to equilibrium. The solution remains in the region where $\psi$ is strictly decreasing.}
    \label{fig-monotoneconv}
\end{figure}

\section{Periodic solutions for strongly excitatory networks}
\label{periodic}

From equation \eqref{mass} we can also construct various periodic solutions for the activity $N(t)$, which generate periodic solutions of the System \eqref{eq1} by means of Theorem~\ref{periodicn}. However, we observe from the integral equation~\eqref{mass} that a $\sigma$-periodic solution satisfies that $\int_{t-\sigma}^t N(s)ds $ is constant and thus $\psi(N(t))$ is also constant. Hence, except when $\psi$ is locally constant,  the only continuous $\sigma$-periodic solutions are constant (steady states). When $\psi'$ changes sign, we build  several  types of solutions including piece-wise constant discontinuous $\sigma$-periodic solutions and piece-wise smooth discontinous $2\sigma$-periodic solutions.

\subsection{Piece-wise constant periodic solutions}
Our first goal is to  build  $\sigma$ periodic solutions with jump discontinuities which keep constant  the value of $\psi(N(t))$.
As a consequence of Theorem~\ref{periodicn}, this is possible when $\psi'$ changes sign and some structure condition is met on $\psi$.

\begin{thm}[\bf Existence of piece-wise constant $\sigma$-periodic activities]
\label{piecewiseperiodic}
Assume that $\psi'$ changes sign. Let
$N_1\neq N_2$ be numbers in $(0,p_\infty]$ such that $\psi(N_1)=\psi(N_2)$. Consider the function defined by
\begin{equation}
\label{Nperiod}
    N(t)=\left\{
    \begin{matrix}
    N_1&\textrm{when}\quad t\in[0,\alpha)\vspace{0.15cm},\\
    N_2&\textrm{when}\quad t\in[\alpha,\sigma),
    \end{matrix}
    \right.
\end{equation}
and assume there is an $\alpha\in(0,\sigma)$ such that
\begin{equation}
\label{algebraicsigma}
    \alpha N_1+(\sigma-\alpha)N_2+\psi(N_1)=1.
\end{equation}
Then the periodic extension of $N(t)$ determines a $\sigma$-periodic solution of System \eqref{eq1}.
\end{thm}
\begin{rmk}
  Notice that, if  such an $\alpha$ exists, there is a steady state
  between $N_1$ and $N_2$, because if $N_1<N_2$ then
  $\sigma N_1+\psi(N_1)<1<\sigma N_2+\psi(N_2)$. Therefore there
  exists $N^*$ such that   $\sigma N^*+\psi(N^*)=1$. If $N_2 <N_1$
  we proceed in the same way exchanging $N_1$ and $N_2$
\end{rmk}
\begin{rmk}
The same construction can be done for a piece-wise constant $N(t)$ with more than one jump, as long as it verifies the equation \eqref{mass} and $\psi$ remains constant for the values taken by $N(t)$.
\end{rmk}

As an example where condition \eqref{algebraicsigma} is verified, consider $\sigma=1$ and $\varphi(N)=\left(1+e^{-9N+3.5}\right)^{-1}$ so that
$$\psi(N)=N(1+e^{-9N+3.5}).$$
\begin{figure}[ht!]
    \centering
    \includegraphics[scale=0.6]{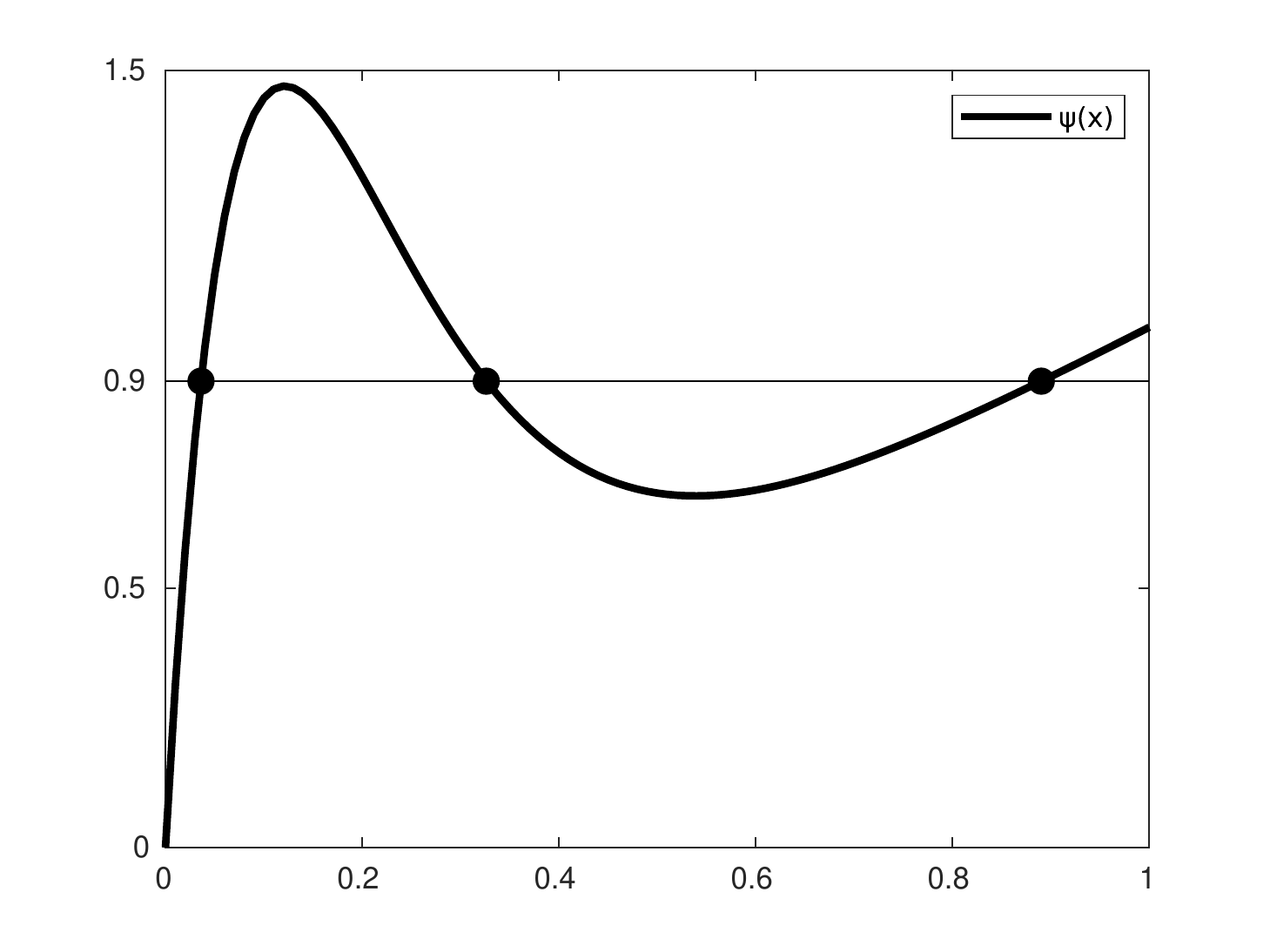}
      \\[-5pt]
    \caption{Graphic of $\psi(u)=u(1+e^{-9u+3.5})$.}
    \label{psi2}
\end{figure}
As we see in Figure \ref{psi2}, $\psi'$ changes sign and there are three solutions of the equation $\psi(u)=0.9$. Take $N_1$ as the minimal solution and $N_2$ as the maximal one, so that there exists
$\alpha:=\frac{N_2-0.1}{N_2-N_1}\in(0,1)$ such that
$$\alpha N_1+(1-\alpha)N_2=0.1,$$
since $N_1<0.1<N_2$. Therefore condition \eqref{algebraicsigma} holds and $N(t)$ defined in \eqref{Nperiod} determines a periodic solution of \eqref{eq1} by means of Theorem~\ref{periodicn}.

We can also get continuous periodic solutions for a specific type of firing rates as we state it in the following proposition.
\begin{prop}
Let $\varphi$ be a smooth function such that $\varphi(u)=Cu$ on some interval $[a,b]\subseteq [0,p_\infty]$ with $C>1$. Assume moreover that $\sigma>0$ satisfies the inequality
\begin{equation}
\label{psiconst}
    a\sigma < 1-\frac{1}{C} < b\sigma.
\end{equation}
Let  $u$ be a bounded  $\sigma$ periodic function such that   $\int_0^\sigma u(s)\,ds=0.$
Then there exists $\alpha >0$ such that 
$$N(t)\coloneqq\frac{1}{\sigma}\left(1-\frac{1}{C}\right)+\alpha u(t)$$
is a solution of \eqref{mass}.
\label{prop:per1}
\end{prop}

\begin{proof} Let  $u$ be a bounded  $\sigma$ periodic function such that   $\int_0^\sigma u(s)\,ds=0.$
For $\alpha>0$ small enough,  $N(t)$ defined in Proposition \ref{prop:per1} is such that $N(\R) \subset [a,b]$, and so, $N$ satisfies \eqref{mass} because $\psi(N(\cdot)) =\frac 1 C$.
\end{proof}

\subsection{Piece-wise monotone $2\sigma$-periodic solutions}
With more elaborate arguments inspired in the works of Hadeler et al. in \cite{hadeler1977periodic} on periodic solutions of di\-ffe\-ren\-tial delay equations, we can also build $2\sigma$-periodic solutions of System \eqref{eq1}, which are piece-wise monotone. In order to achieve the result, we study the delay differential equation given by
\begin{equation}
\label{Nretard}
    \frac{d}{dt}\psi(N(t))=N(t-\sigma)-N(t).
\end{equation}

\begin{thm}[\bf Existence of piece-wise monotone $2\sigma$-periodic solutions]
\label{2sigmaperiod}
Assume that $\psi$ is smooth with $\psi'$ changing sign. Consider $\underline{N}>0$ such that it is a local minimum of $\psi$ and there exists $\varepsilon>0$ such that $\psi$ is strictly convex on $(\underline{N}-\varepsilon,\underline{N}+\varepsilon)$. Then for $\sigma>0$ small enough there exists a $2\sigma$ periodic solution $N(t)$ of \eqref{Nretard} with $\psi(N)\in W^{1,\infty}(\R)$, such that $N(t)$ is strictly decreasing on $(0,2\sigma)$ with a discontinuity at $\sigma$.

If in addition $\psi(\underline{N})<1$ and $\psi(\underline{N}\pm\varepsilon)>1$, then there exists a solution $n$ of equation \eqref{eq1} such that the activity $N(t)$ solves \eqref{mass}.
\end{thm}

\begin{proof}
Since $\psi$ is strictly convex around the local minimum $\underline{N}$, consider $\bar{N}^-,\,\bar{N}^+$ be two positive constants such that $$\psi(\bar{N}^-)=\psi(\bar{N}^+),\quad \bar{N}^-<\underline{N}<\bar{N}^+$$ 
with $\psi$ strictly decreasing on $[\bar{N}^-,\underline{N}]$ and strictly increasing on $[\underline{N},\bar{N}^+]$. We construct a periodic solution $N(t)$ of \eqref{Nretard} satisfying the following conditions:

\begin{enumerate}
    \item \label{1stb}$\left\{\begin{matrix}\underline{N}< N(t)< \bar{N}^+,\:N'(t)<0 &\textrm{for}\:t\in(0,\sigma),\\
    \bar{N}^-< N(t)<\underline{N},\:N'(t)<0&\textrm{for}\:t\in(\sigma,2\sigma).
    \end{matrix}\right.$
    \item \label{2ndb}$N(0^+)=\bar{N}^+,\,N(2\sigma^-)=\bar{N}^-$.
    \item \label{3rdb}$\psi(N(\sigma^-))=\psi(N(\sigma^+)).$
\end{enumerate}

The first step is to build a periodic solution solving equation \eqref{Nretard} on $(0,\sigma)$ and $(\sigma,2\sigma)$, which satisfies conditions \ref{1stb} and \ref{2ndb}. Let $C$ be the closed subset of the Banach space $\C[0,\sigma]$ defined by
\begin{equation*}
    C\coloneqq\left\{N(t)\in\C[0,\sigma]\colon N(0)=\bar{N}^+,\,N(t)\,\textrm{non-increasing},\,\psi(N(t))\ge\psi(\bar{N}^+)-(\bar{N}^+-\bar{N}^-)\sigma\:\textrm{for}\:t\in[0,\sigma]\right\}.
\end{equation*}
Observe that for $\sigma$ small enough we assure that $\psi(N(t))>\psi(\underline{N})$ for all $t\in[0,\sigma]$. Our strategy is to build a solution on $[0,2\sigma]$ such that its restriction to $[0,\sigma]$ is a fixed point of an operator in $C$. 

Now for $N\in C$ we define $M(t)$ as the solution of the backward problem
\begin{equation}
    \left\{\begin{matrix}
    \frac{d}{dt}\psi(M(t))=N(t-\sigma)-M(t),&t\in (\sigma,2\sigma).\\
    M(2\sigma)=\bar{N}^-.&
    \end{matrix}\right.
\end{equation}
For $\sigma$ small enough, this equation is well-posed in the classical sense  on $[\sigma,2\sigma)$. Moreover we have $\psi'(M(t))<0,\,M'(t)<0$ and $\tfrac{d}{dt}\psi(M(t))$ verifies
$$\frac{d}{dt}\psi(M(t))\le \bar{N}^+-\bar{N}^-,\quad\forall t\in(\sigma,2\sigma).$$
  By integrating this inequality, between $t$ and $2\sigma$,
  we get $\psi(M(t))\ge\psi(\bar{N}^+)-(\bar{N}^+-\bar{N}^-)\sigma>\psi(\underline{N})$, and since $M(2\sigma)=\bar{N}^-$ we have $\bar{N}^-< M(t)<\underline{N}$ for $t\in(\sigma,2\sigma)$. Similarly, we define $L(t)$ as the solution of problem
\begin{equation}
    \left\{\begin{matrix}
    \frac{d}{dt}\psi(L(t))=M(t-\sigma)-L(t),&t\in (2\sigma,3\sigma).\\
    L(2\sigma)=\bar{N}^+,&
    \end{matrix}\right.
\end{equation}
so that for $\sigma$ small enough it is well-posed in the classical sense on $[2\sigma,3\sigma)$. Moreover we have $\psi'(L(t))>0,\,L'(t)<0$, and $\psi(L(t))\ge\psi(\bar{N}^+)-(\bar{N}^+-\bar{N}^-)\sigma>\psi(\underline{N})$, which implies that $\underline{N}< L(t)< \bar{N}^+$ for $t\in (2\sigma,3\sigma)$, since $L(2\sigma)=\bar{N}^+$.

Therefore we define the continuous map $T\colon C\to C$ given by $T[N](t)=L(t-2\sigma)$ and we look for a fixed point of the operator $T$ in order to find a $2\sigma$-periodic function $N(t)$ with $\psi(N(0^+))=\psi(N(2\sigma^-))$, satisfying equation \eqref{Nretard} on $(0,\sigma)$ and $(\sigma,2\sigma)$ with jump discontinuities such that $\psi(N(t))$ is continuous.

Now we proceed to prove that $T$ is a contraction for $\sigma$ small enough. Consider $N_1,N_2\in C$ with their respective $M_1,M_2$ and $L_1,L_2$. For the difference between $M_1$ and $M_2$ we have for $t\in[\sigma,2\sigma]$
\begin{equation*}
\begin{split}
    |M_2(t)-M_1(t)|&\le A_\sigma|\psi(M_2(t))-\psi(M_1(t))|\\
                   &\le A_\sigma\left|\int_t^{2\sigma}\frac{d}{dt}\left(\psi(M_1(t'))-\psi(M_2(t'))\right)\,dt'\right|\\
                   &\le\sigma A_\sigma\|N_2-N_1\|_\infty+\sigma A_\sigma\|M_1-M_2\|_\infty,
\end{split}
\end{equation*}
with $A_\sigma\coloneqq\max\left\{\left|\left(\psi^{-1}\right)'(u)\right|\colon \psi(\bar{N}^+)-(\bar{N}^+-\bar{N}^-)\sigma\le u\le\psi(N^+)\right\}$ and $\psi^{-1}$ is the local inverse around $\bar{N}^-$. Since $\sigma A_\sigma\to 0$ when $\sigma\to0$, we deduce the following estimate for $\sigma$ small enough

\begin{equation}
\label{diffM}
    \|M_2-M_1\|_\infty\le \frac{\sigma A_\sigma}{1-\sigma A_\sigma}\|N_2-N_1\|_\infty.
\end{equation}

Analogously for the difference between $L_1$ and $L_2$ we get
\begin{equation}
\label{diffL}
    \|L_2-L_1\|_\infty\le \frac{\sigma B_\sigma}{1-\sigma B_\sigma}\|M_2-M_1\|_\infty,
\end{equation}
with $B_\sigma\coloneqq\max\left\{\left|\left(\psi^{-1}\right)'(u)\right|\colon \psi(\bar{N}^+)-(\bar{N}^+-\bar{N}^-)\sigma\le u\le\psi(N^+)\right\}$ and $\psi^{-1}$ is now considered as the local inverse around $\bar{N}^+$. Therefore we conclude from estimates \eqref{diffM} and \eqref{diffL} that $T$ a contraction and we get a unique $2\sigma$-periodic function $N(t)$ satisfying the conditions \ref{1stb} and \ref{2ndb} and solving \eqref{Nretard} on $(0,\sigma)$ and $(\sigma,2\sigma)$.

The next step is to prove that the constructed solution $N(t)$ verifies the condition \ref{3rdb}. From equation \eqref{Nretard} we deduce that
$$Q(t)\coloneqq\int_{t-\sigma}^{t} N(s)\,ds+\psi(N(t))$$
is piece-wise constant and we get the following equalities
\begin{equation}
    \begin{split}
        \psi(N(\sigma^-))-\psi(N(0^+))=\int_\sigma^{2\sigma} N(s)\,ds-\int_0^\sigma N(s)\,ds,\\
        \psi(N(2\sigma^-))-\psi(N(\sigma^+))=\int_0^\sigma N(s)\,ds-\int_\sigma^{2\sigma} N(s)\,ds.
        \end{split}
\end{equation}
Since $\psi(N(0^+))=\psi(\bar{N}^+)=\psi(\bar{N}^-)=\psi(N(2\sigma^-))$, we conclude that $\psi(N(\sigma^-))=\psi(N(\sigma^+))$. Moreover, we conclude that $\psi(N(t))$ is absolutely continuous and thus $Q(t)$ is constant and given by
$$Q(t)\equiv Q[\bar{N}^+]\coloneqq\psi(\bar{N}^+)+\int_\sigma^{2\sigma}N(s)\,ds.$$
This proves the first part of theorem.

Assume now the additional hypothesis $\psi(\underline{N})<1$ and $\psi(\underline{N}\pm\varepsilon)>1$, thus there exist two pairs $(\bar{N}^-_1,\bar{N}^+_1),\,(\bar{N}^-_2,\bar{N}^+_2)$ of positive numbers such that for $\sigma$ small enough, the following conditions hold
\begin{itemize}
    \item[$\bullet$] $\psi(\bar{N}^-_1)=\psi(\bar{N}^+_1)$ and $\psi(\bar{N}^-_2)=\psi(\bar{N}^+_2)$.
    \item[$\bullet$] $\bar{N}^-_1<\bar{N}^-_2<\underline{N}<\bar{N}^+_2<\bar{N}^+_1$.
    \item[$\bullet$] $\psi(\bar{N}^+_1)>1-\sigma\bar{N}^-_1$ and $\psi(\bar{N}^-_2)<1-\sigma\bar{N}^+_2$.
\end{itemize}

By applying Theorem~\ref{periodicn} we can construct two periodic solutions of the System \eqref{eq1} with respective masses $Q[\bar{N}^+_1]$ and $Q[\bar{N}^+_2]$, so we need to find a constant $N^+_0>0$ such that $Q[N^+_0]=1$. Observe we have the following inequalities
\begin{equation}
\begin{split}
        Q[\bar{N}^+_1]>\psi(\bar{N}^+_1)+\sigma\bar{N}^-_1>1\\
        Q[\bar{N}^+_2]<\psi(\bar{N}^+_2)+\sigma\bar{N}^+_2<1.
\end{split}
\end{equation}
Since $Q[\cdot]$ is continuous with respect to the variable $\bar{N}^+$, we conclude by applying the intermediate value theorem the existence of $N^-_0\in (\bar{N}^-_1,\bar{N}^-_2)$ and $N^+_0\in(\bar{N}^+_2,\bar{N}^+_1)$ such that $Q[N^+_0]=1$. Therefore the corresponding periodic solution $N(t)$ of the delay equation \eqref{Nretard} satisfies the equation \eqref{mass} and the conditions \ref{1stb}, \ref{2ndb} and \ref{3rdb}.
\end{proof}

\begin{figure}[ht!]
    \centering
    \includegraphics[scale=0.6]{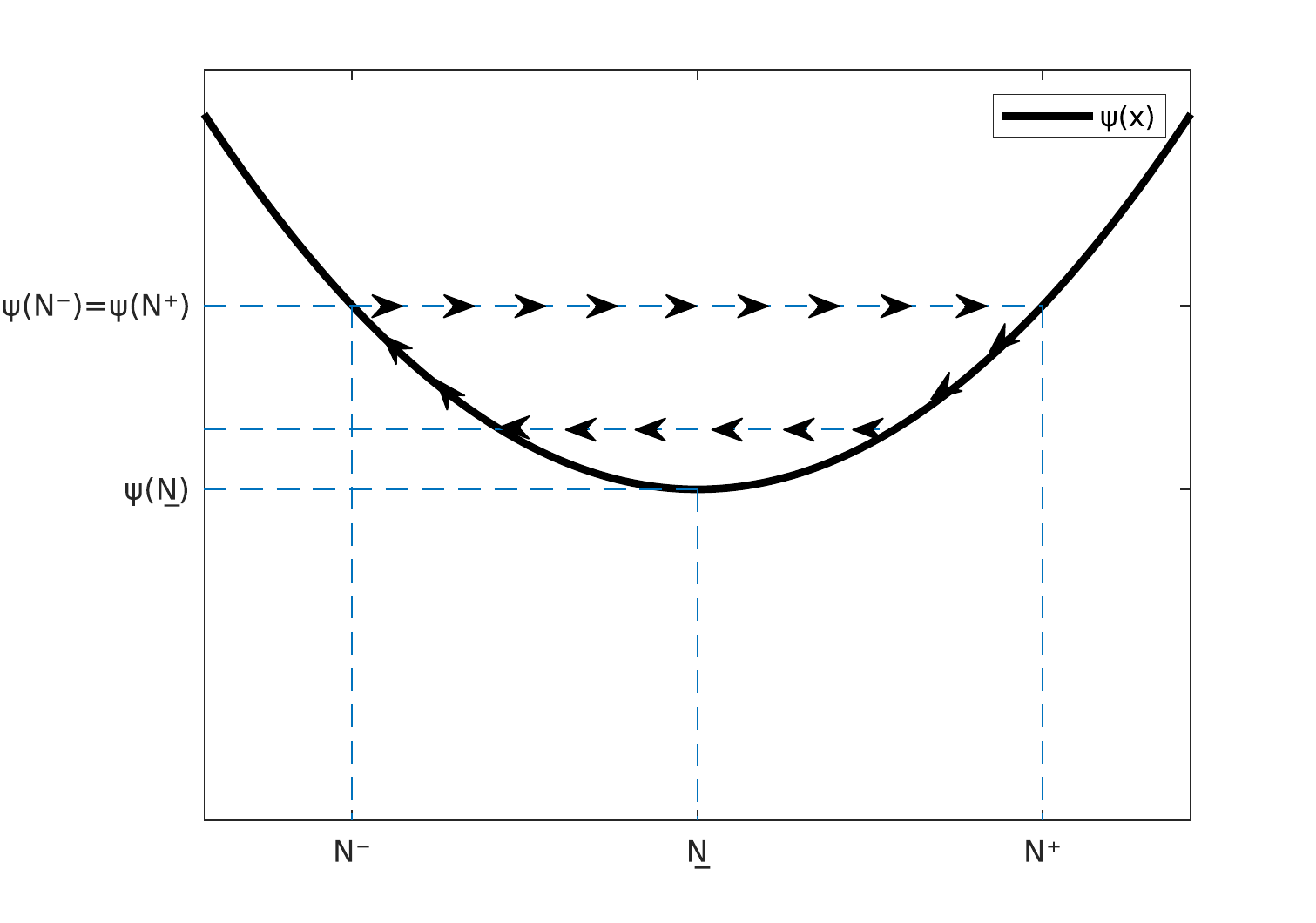}
        \\[-5pt]
    \caption{An example of $\psi$ with a local minimum at $\underline{N}$. The trajectory in arrows over the graph of $\psi$ indicates the variations of periodic solution $N(t)$ and the horizontal arrows indicate a jump discontinuity, which preserves the value of $\psi(N)$ at these jumps.} 
\end{figure}

\begin{rmk}
We can also construct $2\sigma$-periodic solutions $N(t)$ of \eqref{Nretard} around a local maximum of $\psi$, which are piece-wise strictly increasing and preserve the value of $\psi$ at jump discontinuities.
\end{rmk}

\section{Numerical simulations}
\label{numerical}

In order to illustrate the theoretical results of the previous sections, we present numerical results for different networks with multiple steady states, with different types of convergence to equilibrium, and with periodic solutions with jump discontinuities. The numerical illustrations we present below are obtained by solving the equation~\eqref{eq1} with a classical first-order upwind scheme. 

\subsection{Example 1: Convergence to different steady states}
Our first example is a numerical simulation with
multiple steady states. For this example we choose
\begin{equation}
\label{phi1}
    \varphi(N)=\frac{1}{1+e^{-9N+3.5}},\quad\sigma=\frac{1}{2}.
\end{equation}

\begin{figure}[ht!]
    \centering
    \includegraphics[scale=0.59]{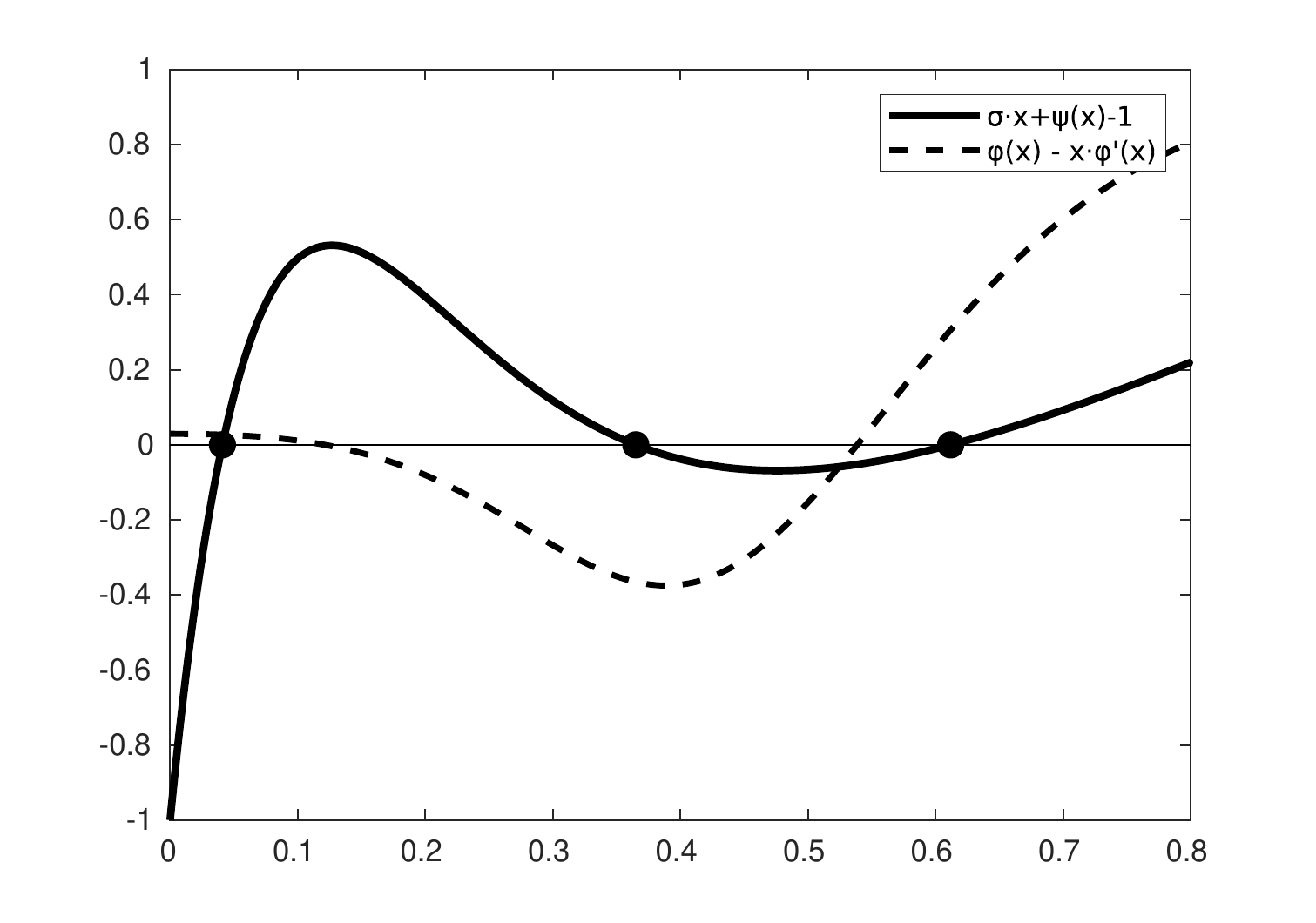}
    \\[-5pt]
    \caption{
      Graphical representation of the three equilibria and
      the sign of $\psi'$, for examples 1 and 2. Two of them are in the region
      where $\psi'>0$ and the other one is the region where
      $\psi'<0$.}
    \label{3equi}
\end{figure}
In this case $\psi'$ changes sign twice since $\varphi(u)-u\varphi'(u)$ does and from equation \eqref{neq} we get three steady states given by $N_*^1\approx 0.0410,\,N_*^2\approx 0.3650$ and $N_*^3\approx 0.6118$, as we observe in Figure \ref{3equi}.
\begin{figure}[ht!]
    \centering
    \begin{subfigure}{0.4\textwidth}
		\includegraphics[width=\textwidth]{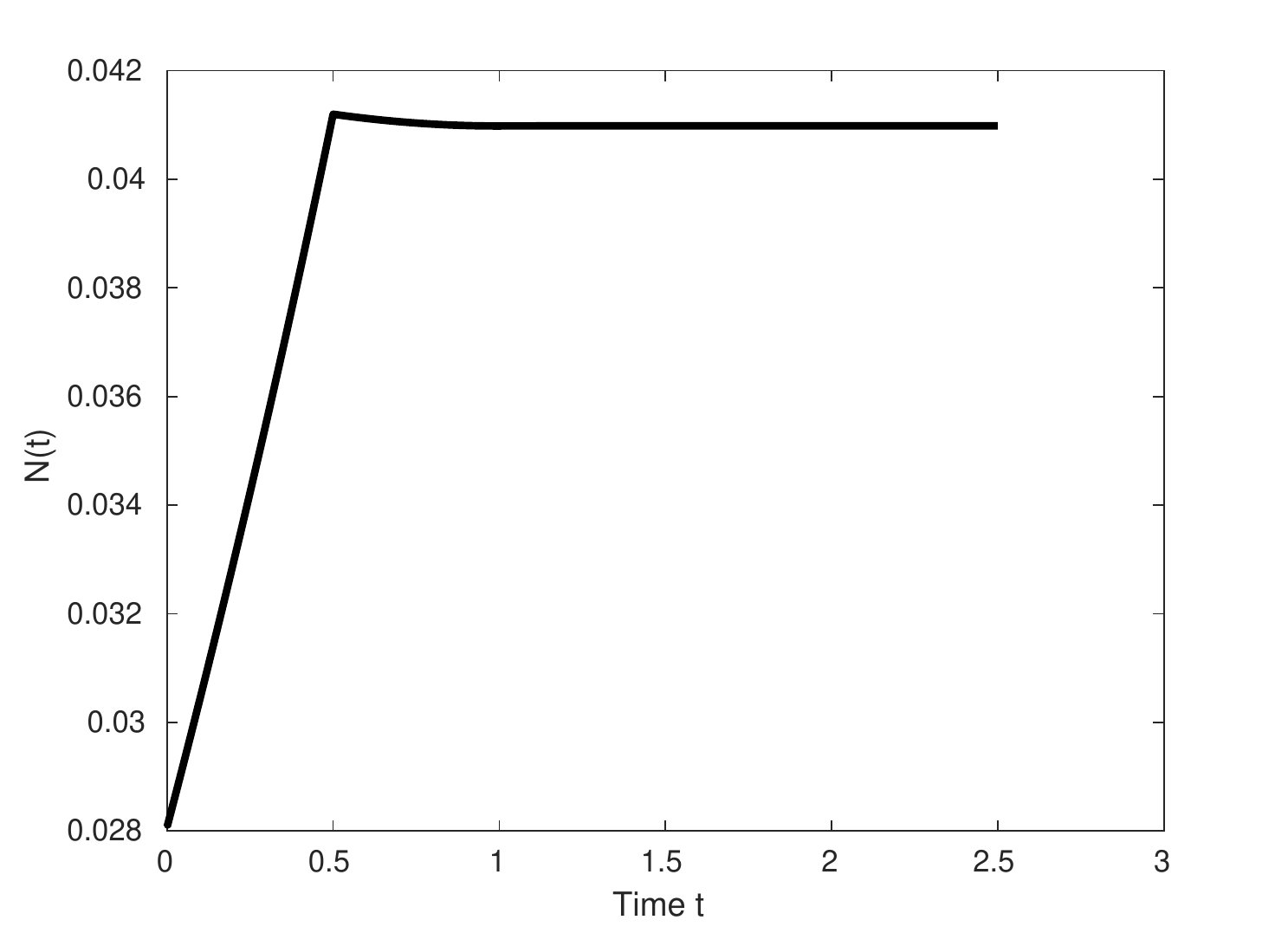}
		 \caption{Activity $N(t)$ for $N(0)=N_0^1$.}
		\label{CaseN1}
	\end{subfigure}
	\begin{subfigure}{0.4\textwidth}
		\includegraphics[width=\textwidth]{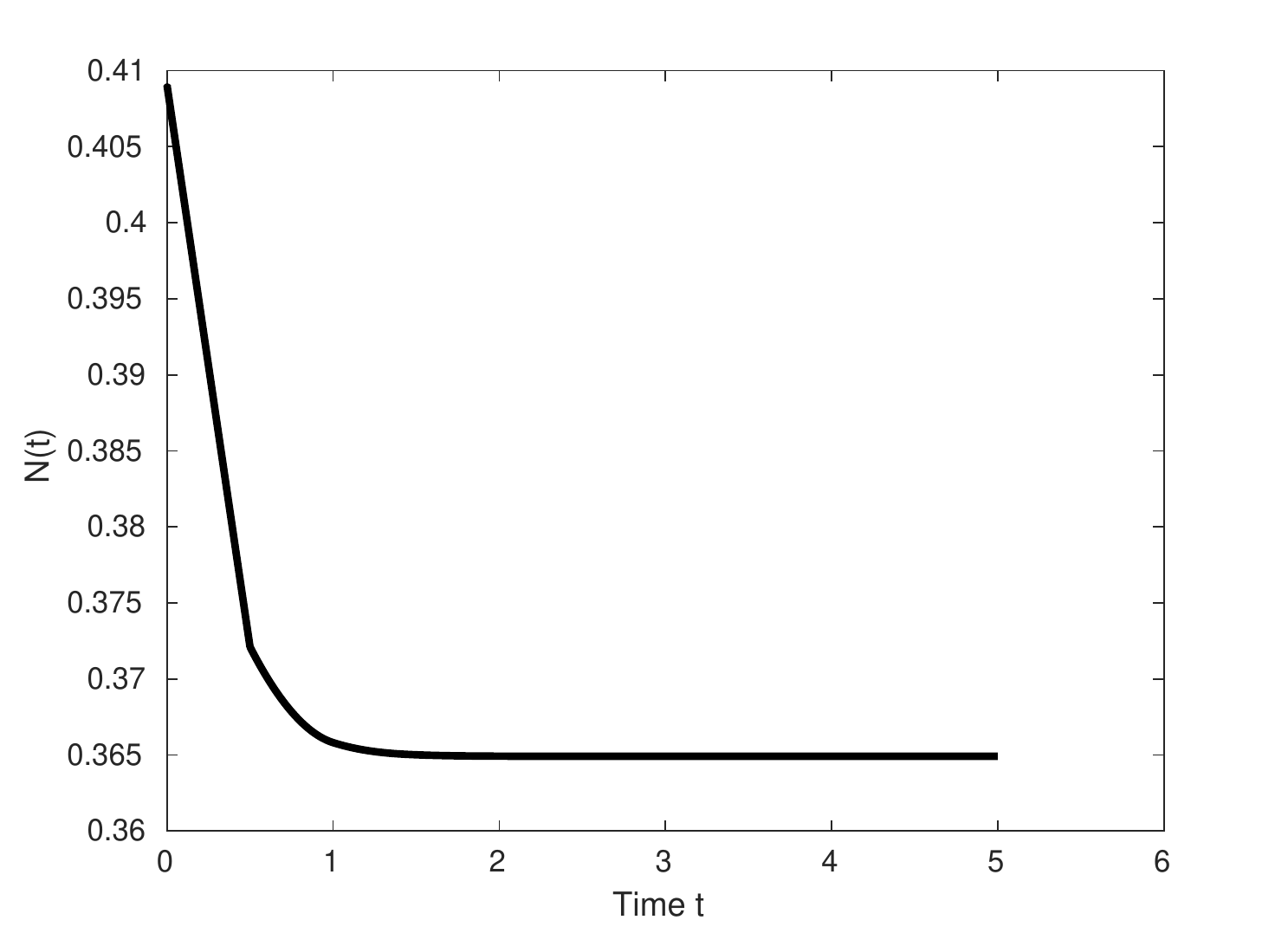}
		 \caption{Activity $N(t)$ for $N(0)=N_0^2$.}
		\label{CaseN2}
	\end{subfigure}
		\begin{subfigure}{0.4\textwidth}
		\includegraphics[width=\textwidth]{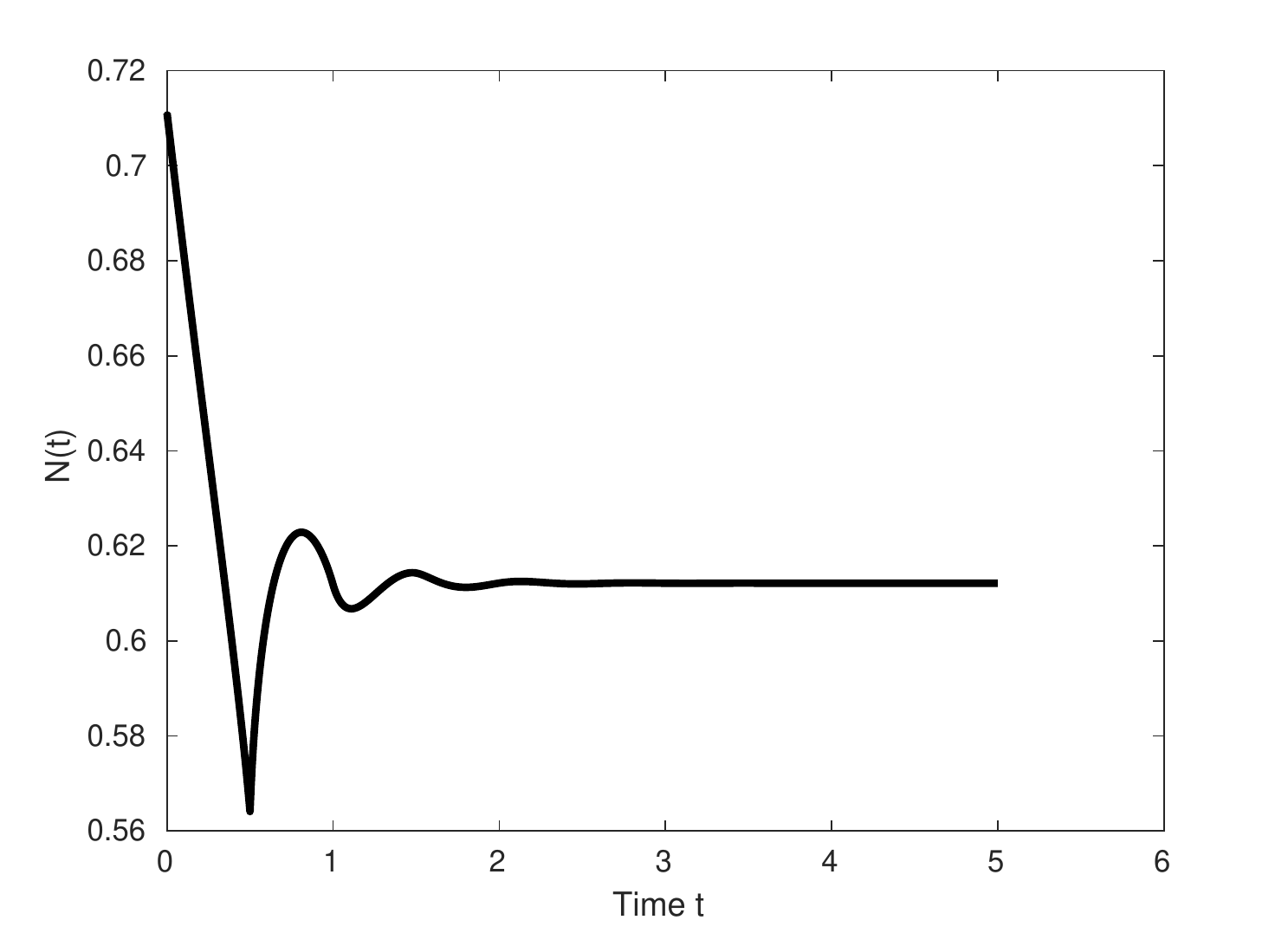}
		 \caption{Activity $N(t)$ for $N(0)=N_0^3$.}
		\label{CaseN3}
	        \end{subfigure}
	          \\[-5pt]
                \caption{{\bf Example 1.}
                  Activities of different solutions with
                  $n_0(s)=\frac{1}{2}e^{-(s-1)_+}$ for different
                values of $N(0)$.}
\end{figure}

When we take $n_0(s)=\frac{1}{2}e^{-(s-1)_+}$ as the initial data, we have three different solutions for $N(0)$ determined by the equation
\begin{equation}
\label{multipleN0}
N(0)=\varphi(N(0))\int_\sigma^\infty n_0(s)\,ds,    
\end{equation}
that are given by $N_0^1\approx 0.0281,\,N_0^2\approx 0.4089,\,N_0^3\approx 0.7114$. These values determine three different branches of solutions, which numerically converge to their respective steady states.
In Figure \ref{CaseN1} we observe that $N(t)$ is increasing in $[0,\sigma]$ and then approaches to the value $N_*^1$, which corresponds to a convergence to equilibrium according to Theorem~\ref{convthm}. Moreover we observe in Figure \ref{CaseN2} that $N(t)$ converges monotonically to $N_*^2$, which satisfies $\psi'(N_*^2)<0$. This is compatible with Theorem~\ref{monotoneconv} in the case when $\psi'(N(t))$ remains negative for all $t\ge0$. Finally in Figure~\ref{CaseN3} we observe that $N(t)$ converges to $N_*^3$ in the same way stated in Theorem~\ref{convthm} with $\psi'(N_*^3)>0$.

\subsection{Example 2: Possible jump discontinuities}
Under the same $\varphi$ defined in \eqref{phi1}, consider now as initial data  $n_0(s)=e^{-(s-0.5)}\mathds{1}_{\{s>0.5\}}$. In this case we also get three possible solutions for $N(0)$ in Equation \eqref{multipleN0}, which are given by $N_0^1\approx 0.0423,\,N_0^2\approx 0.2887,\,N_0^3\approx 0.9958.$

\begin{figure}[ht!]
    \centering
    \begin{subfigure}{0.4\textwidth}
		\includegraphics[width=\textwidth]{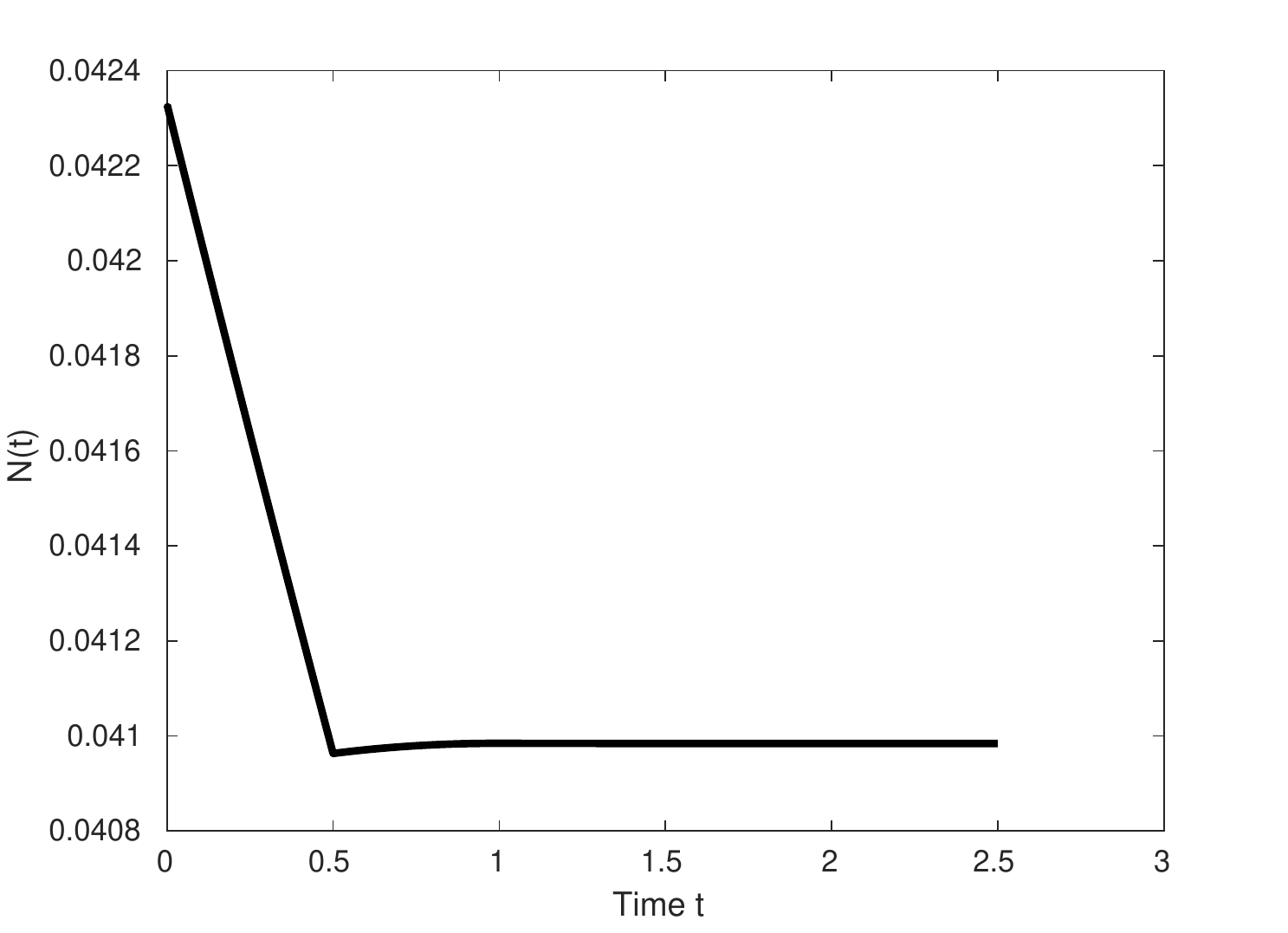}
		 \caption{Activity $N(t)$ for $N(0)=N_0^1$.}
		\label{CaseJN1}
	\end{subfigure}
	\begin{subfigure}{0.4\textwidth}
		\includegraphics[width=\textwidth]{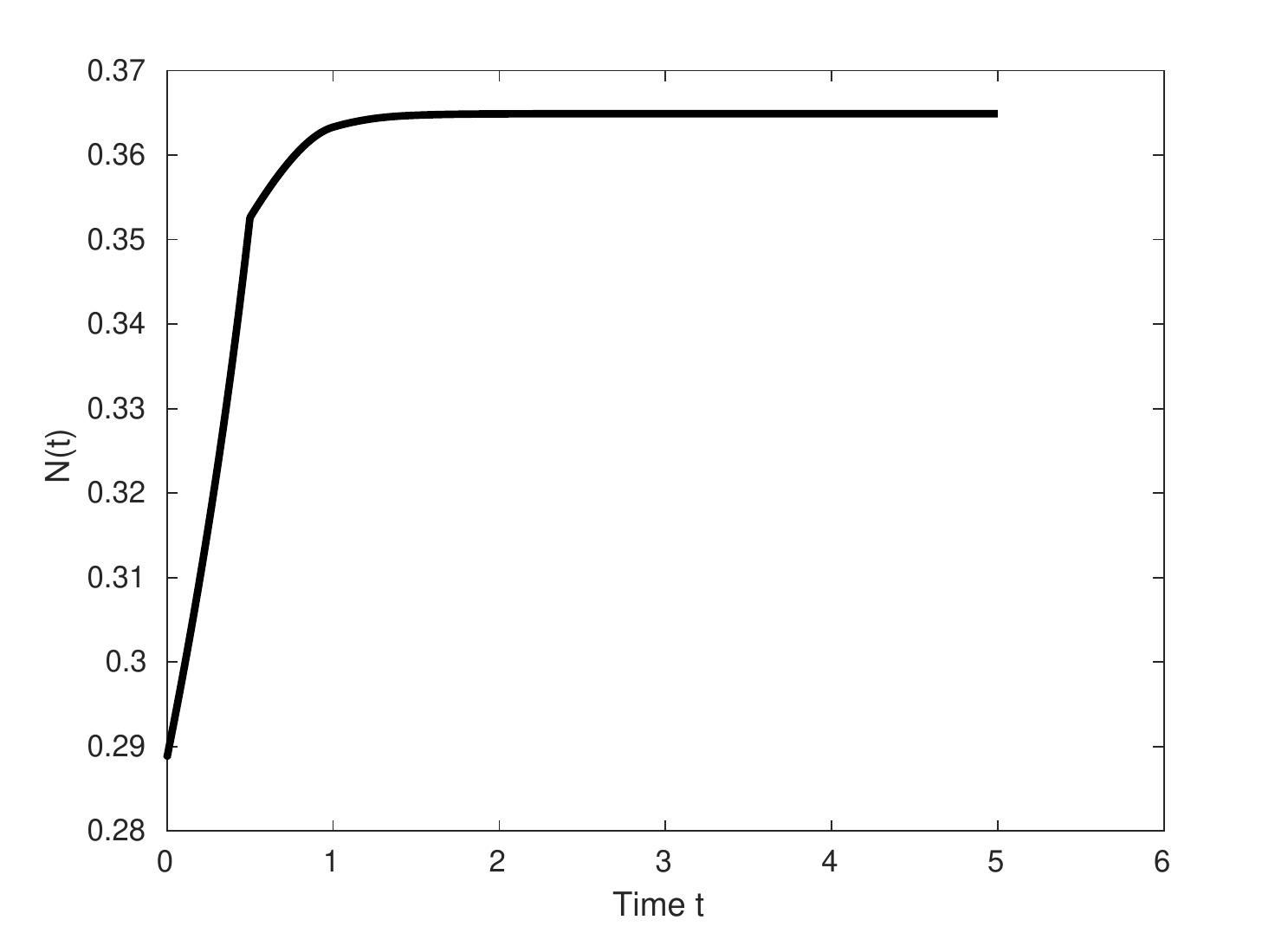}
		 \caption{Activity $N(t)$ for $N(0)=N_0^2$.}
		\label{CaseJN2}
	\end{subfigure}
	\begin{subfigure}{0.4\textwidth}
		\includegraphics[width=\textwidth]{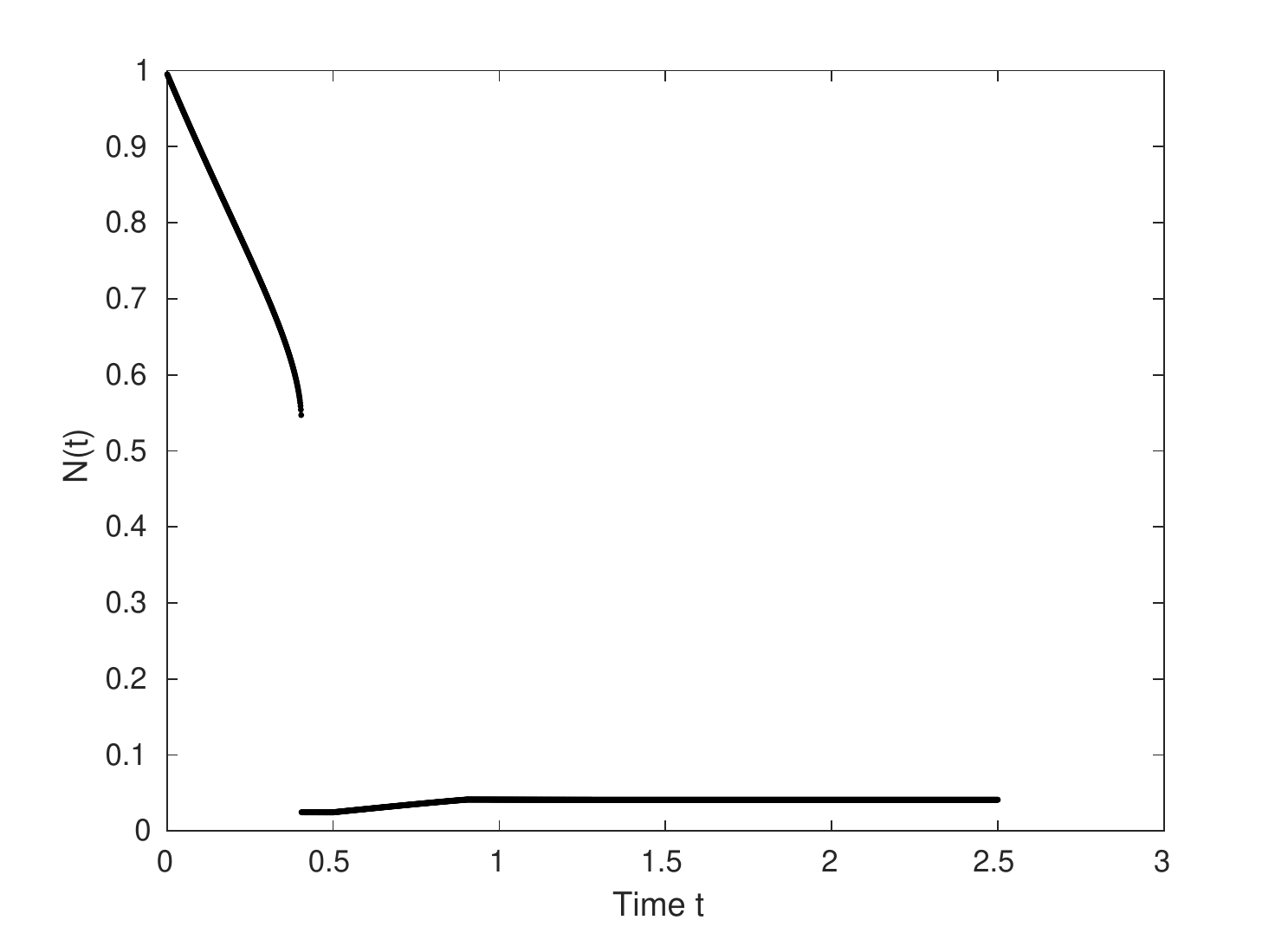}
		 \caption{Activity $N(t)$ for $N(0)=N_0^3$.}
		\label{CaseJN3}
	\end{subfigure}
	\begin{subfigure}{0.4\textwidth}
		\includegraphics[width=\textwidth]{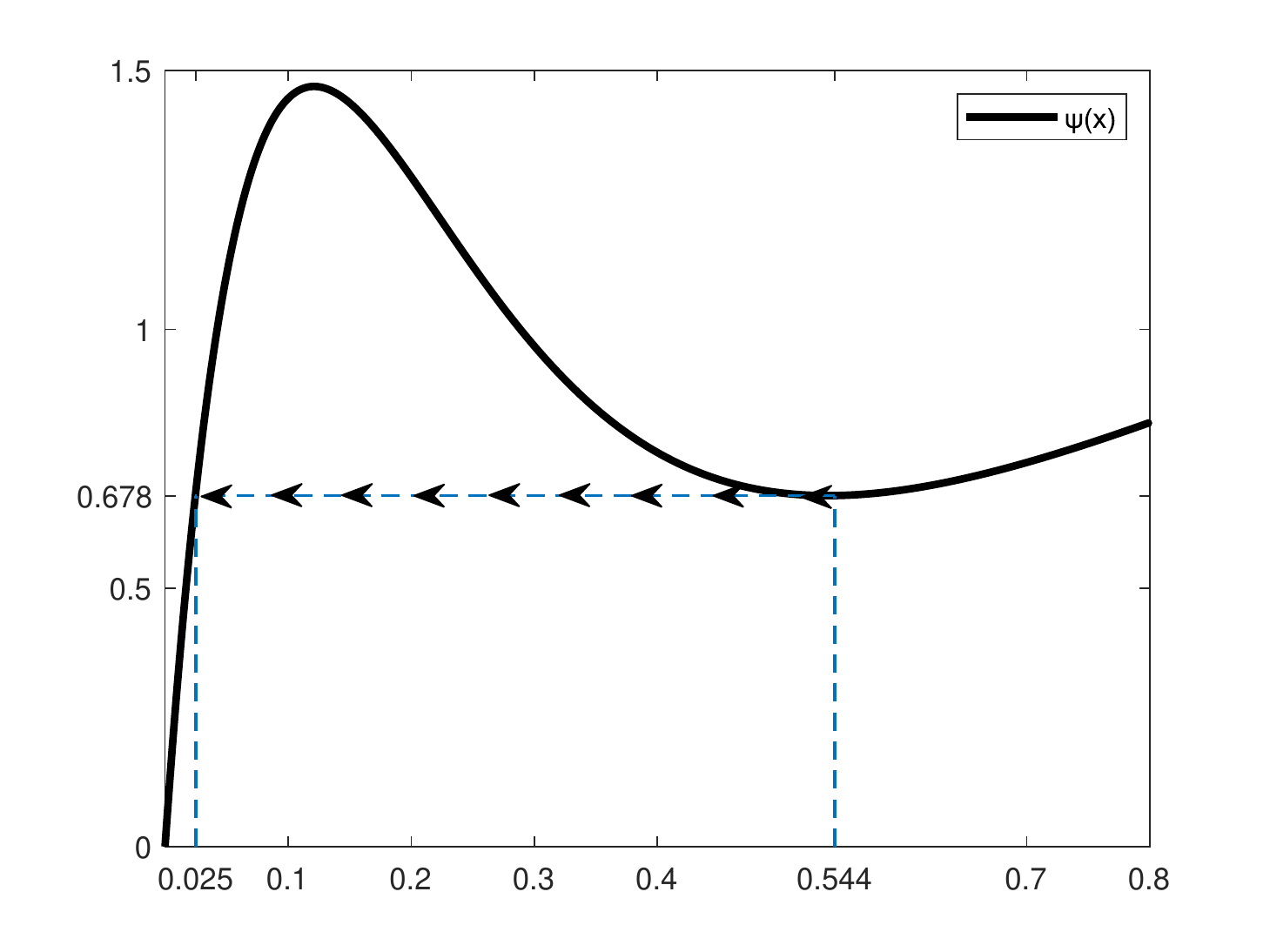}
		\caption{Jump discontinuity for the branch  of
                  $N(0)=N_0^3$ along the graph of $\psi$.}
        \label{Simplejump}
	\end{subfigure}
	  \\[-5pt]
        \caption{{\bf Example 2.}
          Activities of different solutions with
          $n_0(s)=e^{-(s-0.5)}\mathds{1}_{\{s>0.5\}}$
        for different
                values of $N(0)$.}
\end{figure}
In Figure~\ref{CaseJN1} we observe that $N(t)$ is decreasing in $[0,\sigma]$ and then approaches to the value $N_*^1$, which corresponds again to a convergence to equilibrium according to Theorem~\ref{convthm}. In Figure \ref{CaseJN2} we observe that $N(t)$ converges monotonically to $N_*^2$, which satisfies $\psi'(N_*^2)<0$. In this case the solution is increasing and it corresponds to the behaviour stated in Theorem~\ref{monotoneconv}. Moreover in Figure \ref{CaseJN3} we observe that $N(t)$ has a jump discontinuity at some $t_0\in(0,\sigma)$ that causes the solution to change to the branch of $N_0^1$ and then $N(t)$ converges to $N_*^1$ afterwards.  At the jump time, the solution preserves the value of $\psi$ as we show in Figure \ref{Simplejump}. The horizontal arrows represent the change of $N(t)$ along the graph of $\psi$ at this discontinuity.

\subsection{Example 3: Periodic solutions}
To describe periodic solutions we simulate two different examples.

{\em Example 3.1.} 
We consider the firing rat determined by
$$\varphi(N)=\max\{\min(1.6\,N,1),0.25\},\quad \sigma=1.$$
From equation \eqref{neq}, there exists a unique steady state with $N_*=0.375$ and $\psi'(N_*)=0$. Moreover, we observe in Figure \ref{CasePN1} that the solution with initial data $n_0(s)=e^{-s}$ converges to a $\sigma$-periodic solution of \eqref{mass}, which is piece-wise constant whose values oscillate between $N_1=\frac{0.25}{1.6}$ and $N_2=\frac{1}{1.6}$. This periodic profile is an example of the type of solutions presented in Theorem~\ref{piecewiseperiodic}.

\begin{figure}[ht!]
    \centering
    \includegraphics[scale=0.5]{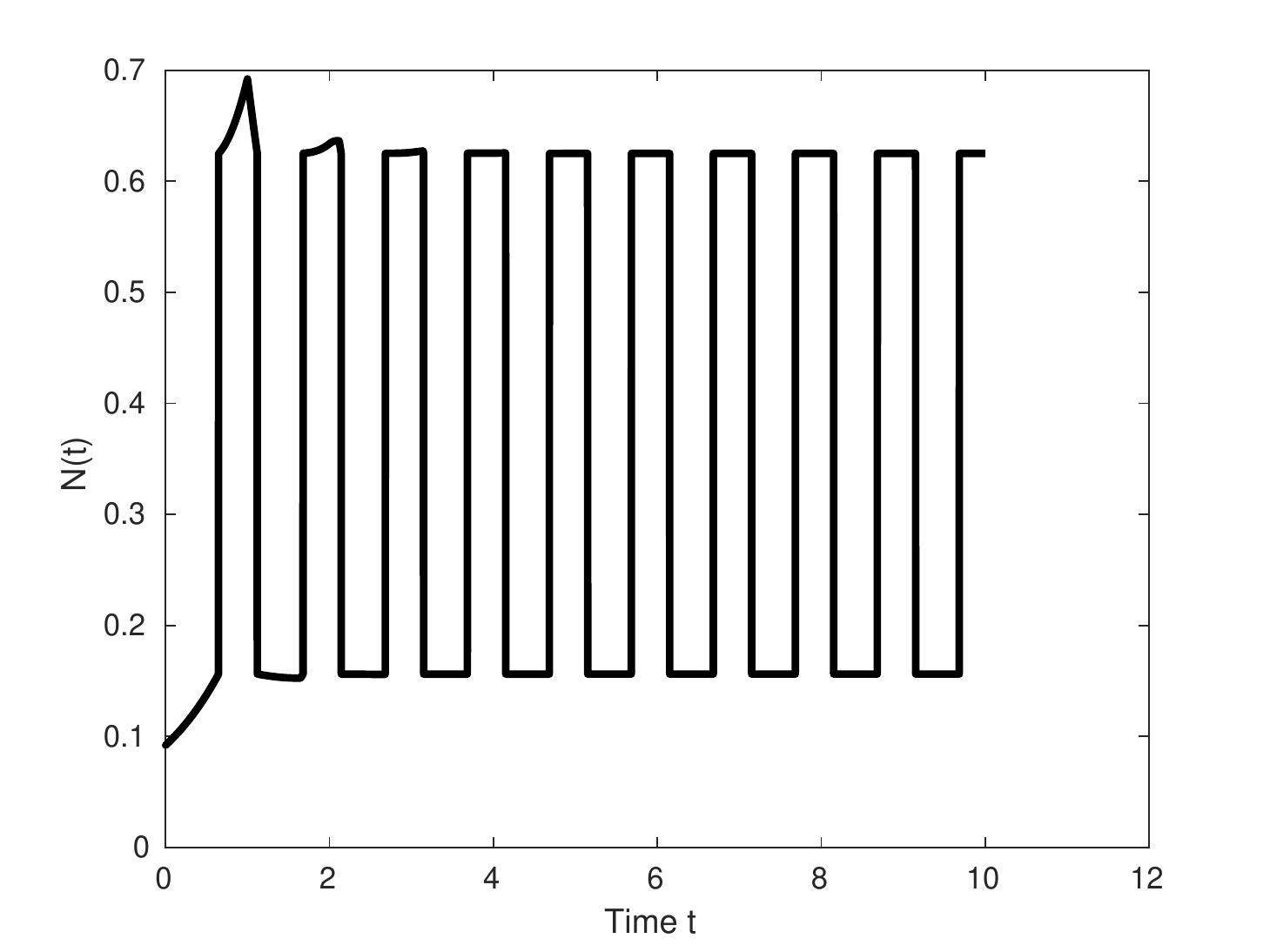}
      \\[-5pt]
    \caption{{\bf Example 3.1.} Activity $N(t)$ for $n_0(s)=e^{-s}$.}
    \label{CasePN1}  
\end{figure}

{\em Example 3.2.}
Next, we consider the firing rate determined by
$$\varphi(N)=\frac{10 N^2}{N^2+1}+0.5,\quad\sigma=1,$$
From equation \eqref{neq} there exists a unique steady state with $N_*\approx0.8186$ and $\psi'(N_*)<0$. For the initial data $n_0(s)=e^{-(s-1)}\mathds{1}_{\{s>1\}}$, we observe in Figure \ref{CasePN2} that the solution is asymptotic to a periodic pattern with jump discontinuities. The period is larger than $1$ since $\psi(N(t))$ is not converging to a constant as we see in Figure \ref{PsiCasePN2}.

\begin{figure}[ht!]
    \centering
    \begin{subfigure}{0.4\textwidth}
    \includegraphics[width=\textwidth]{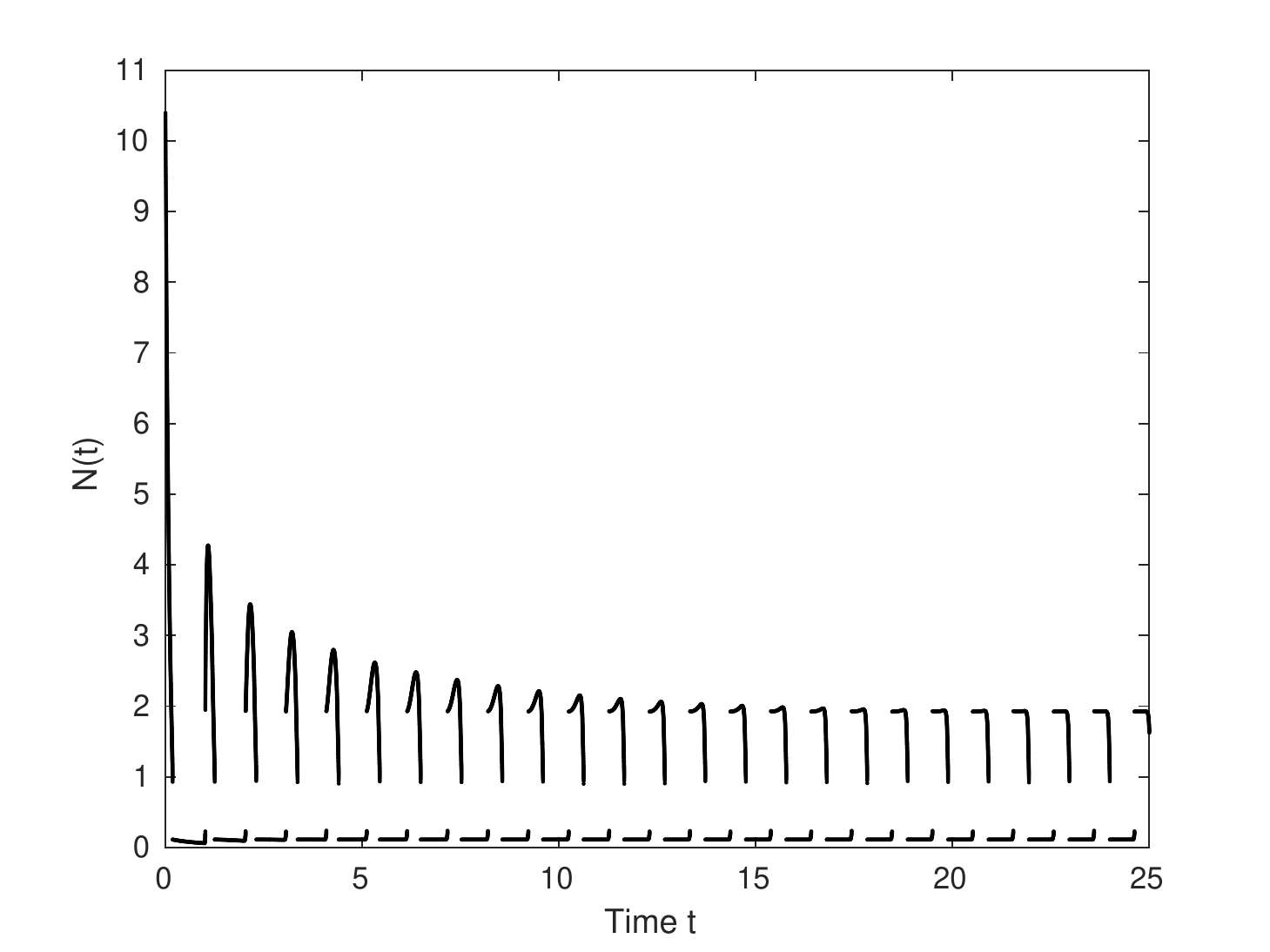}
    \caption{Activity $N(t)$.}
    \label{CasePN2}
    \end{subfigure}
    \begin{subfigure}{0.4\textwidth}
    \includegraphics[width=\textwidth]{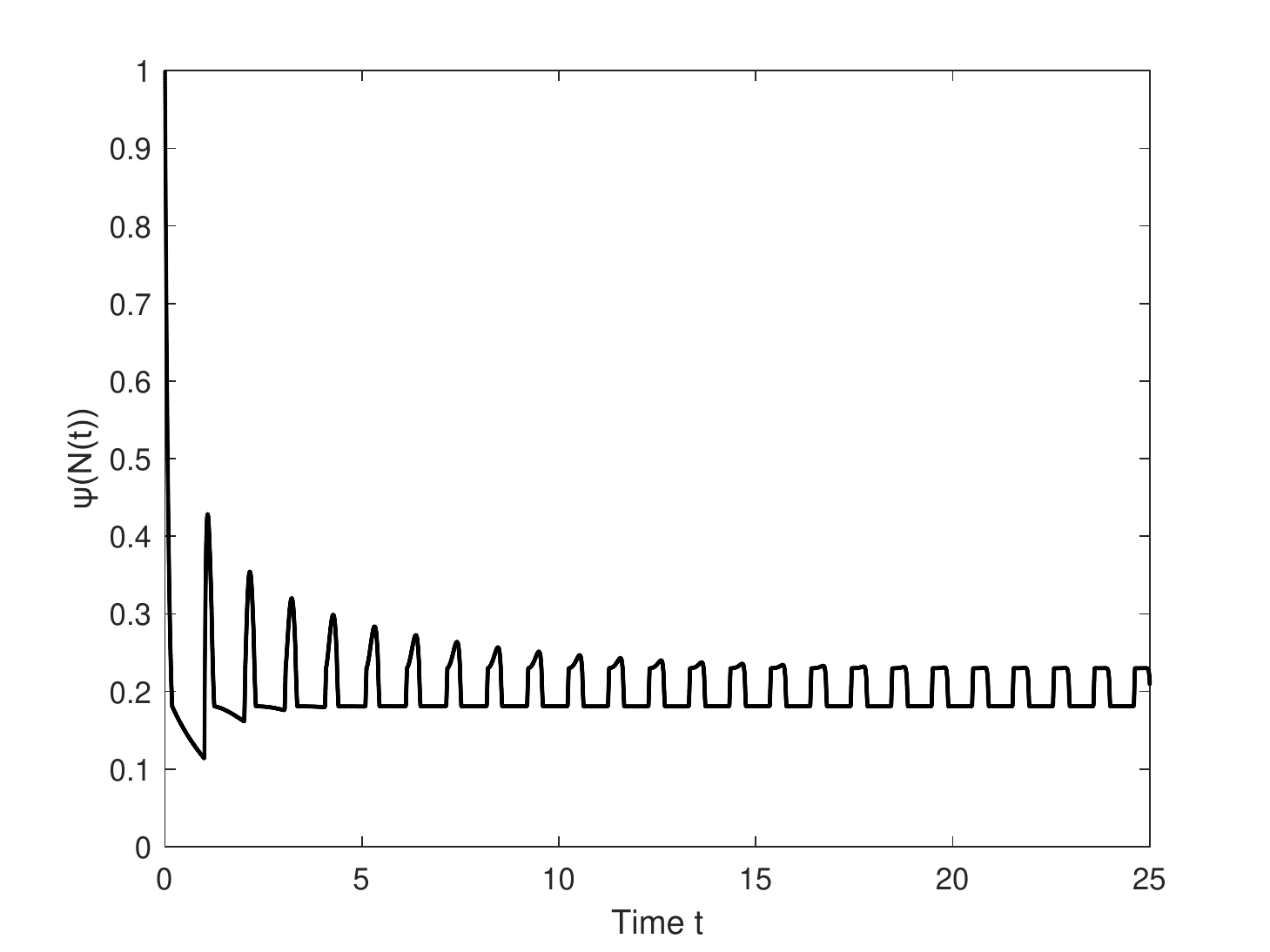}
    \caption{Graph of $\psi(N(t))$.}
    \label{PsiCasePN2}
    \end{subfigure}
      \\[-5pt]
    \caption{{\bf Example 3.2.} Behavior for $n_0(s)=e^{-(s-1)}\mathds{1}_{\{s>1\}}$.}
\end{figure}

\subsection{Example 4: A non monotone firing rate}
Since dynamics in Equation \eqref{eq1} depend heavily on the function $\psi$, theoretical results are valid not only in the strictly excitatory or inhibitory case. This allows to include non monotone examples of functions $\varphi$, which represents a more realistic assumption in the model.

For this example we choose the firing rate determined by
$$\varphi(N)=8e^{-(N-0.1)^2}+8e^{-(N-3)^2},\quad\sigma=\frac{1}{5}.$$
Unlike of previous examples, the function $\varphi$ is non monotone as we see in Figure \ref{Nonmonotone}. 
\begin{figure}[ht!]
    \centering
    \includegraphics[scale=0.58]{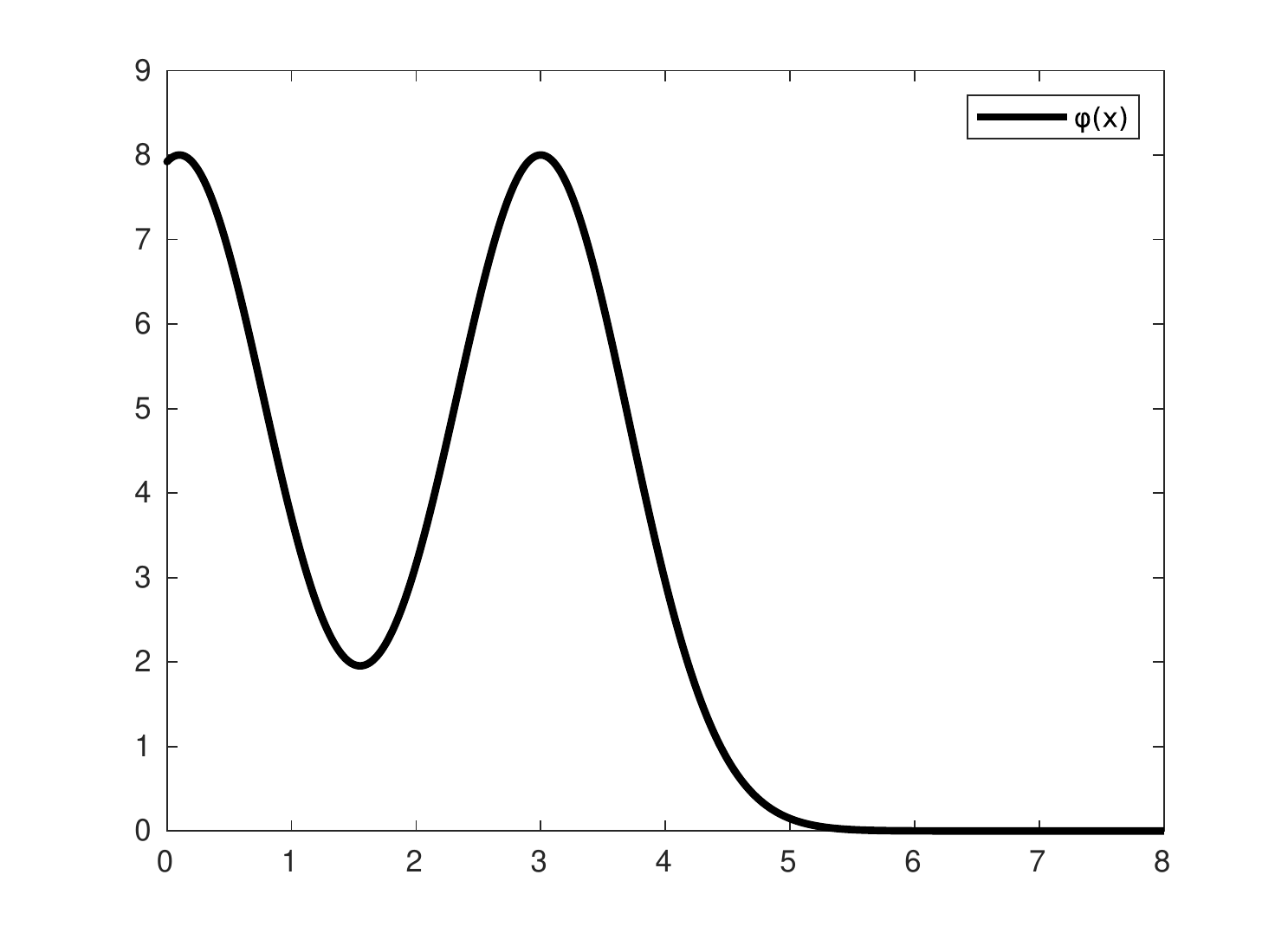}
      \\[-5pt]
    \caption{ Graphical representation of $\varphi$, which changes two times from a excitatory regime to an inhibitory one.}
    \label{Nonmonotone}  
\end{figure}

In this case $\psi'$ changes sign and there exists three steady states with $N_*^1\approx 1.4423,\,N_*^2\approx2.0695,\,N_*^3\approx3.0711$, as we note in Figure \ref{PsiCaseNonMon}.

\begin{figure}[ht!]
    \centering
    \includegraphics[scale=0.58]{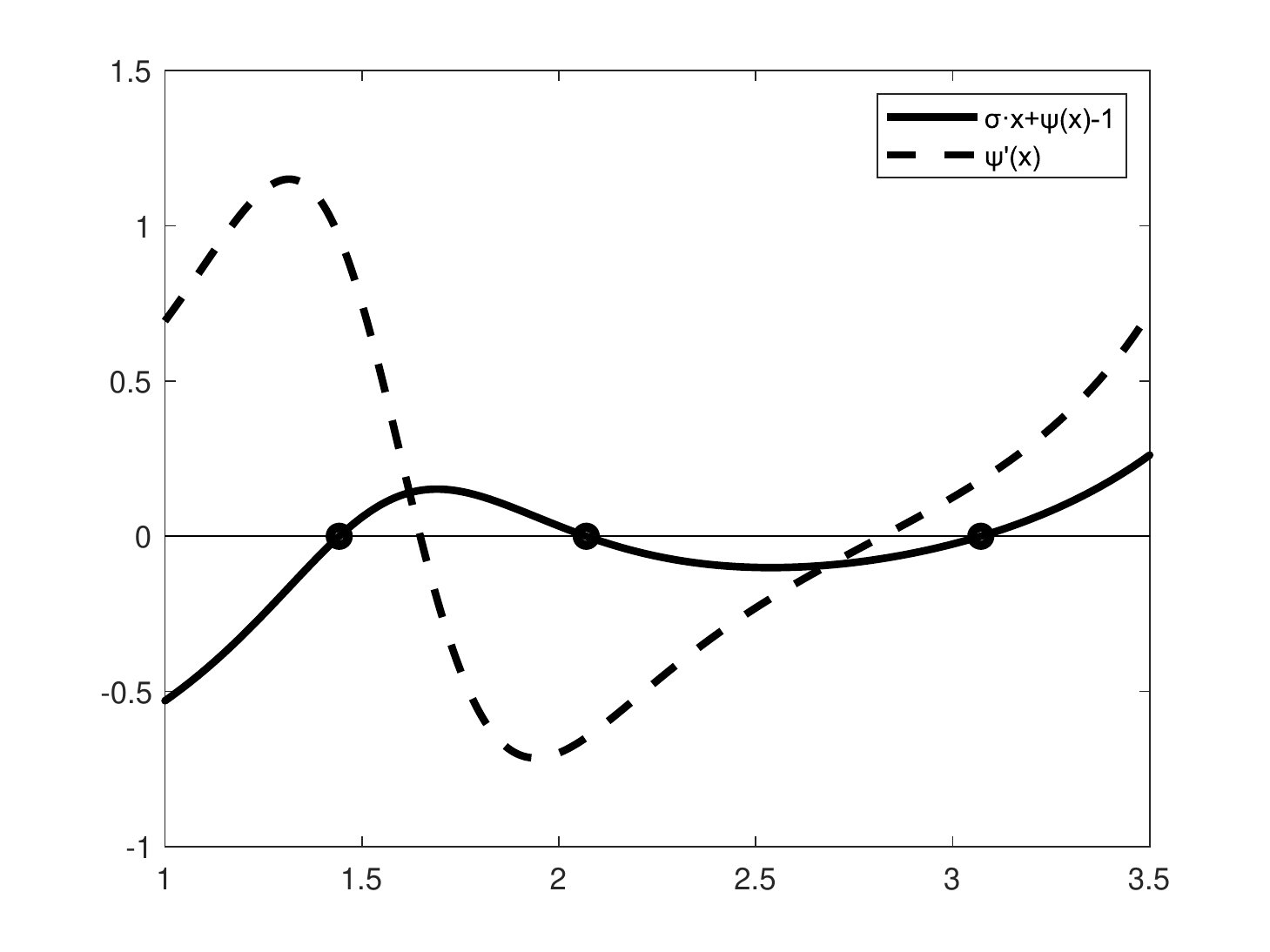}
      \\[-5pt]
    \caption{
      Graphical representation of the three equilibria and
      the sign of $\psi'$, for example 4. Two of them are in the region
      where $\psi'>0$ and the other one is the region where
      $\psi'<0$.}
    \label{PsiCaseNonMon}  
\end{figure}

When we take as initial data $n_0(s)=\frac{2}{3}(1+\cos(s))e^{-s}$ we get three possible solutions for $N(0)$ in Equation \eqref{multipleN0}, which are given by $N_0^1\approx 1.4976,\,N_0^2\approx 1.8163,\,N_0^3\approx 3.7037$. These values determine three different branches of solutions.
\begin{figure}[ht!]
    \centering
    \begin{subfigure}{0.4\textwidth}
		\includegraphics[width=\textwidth]{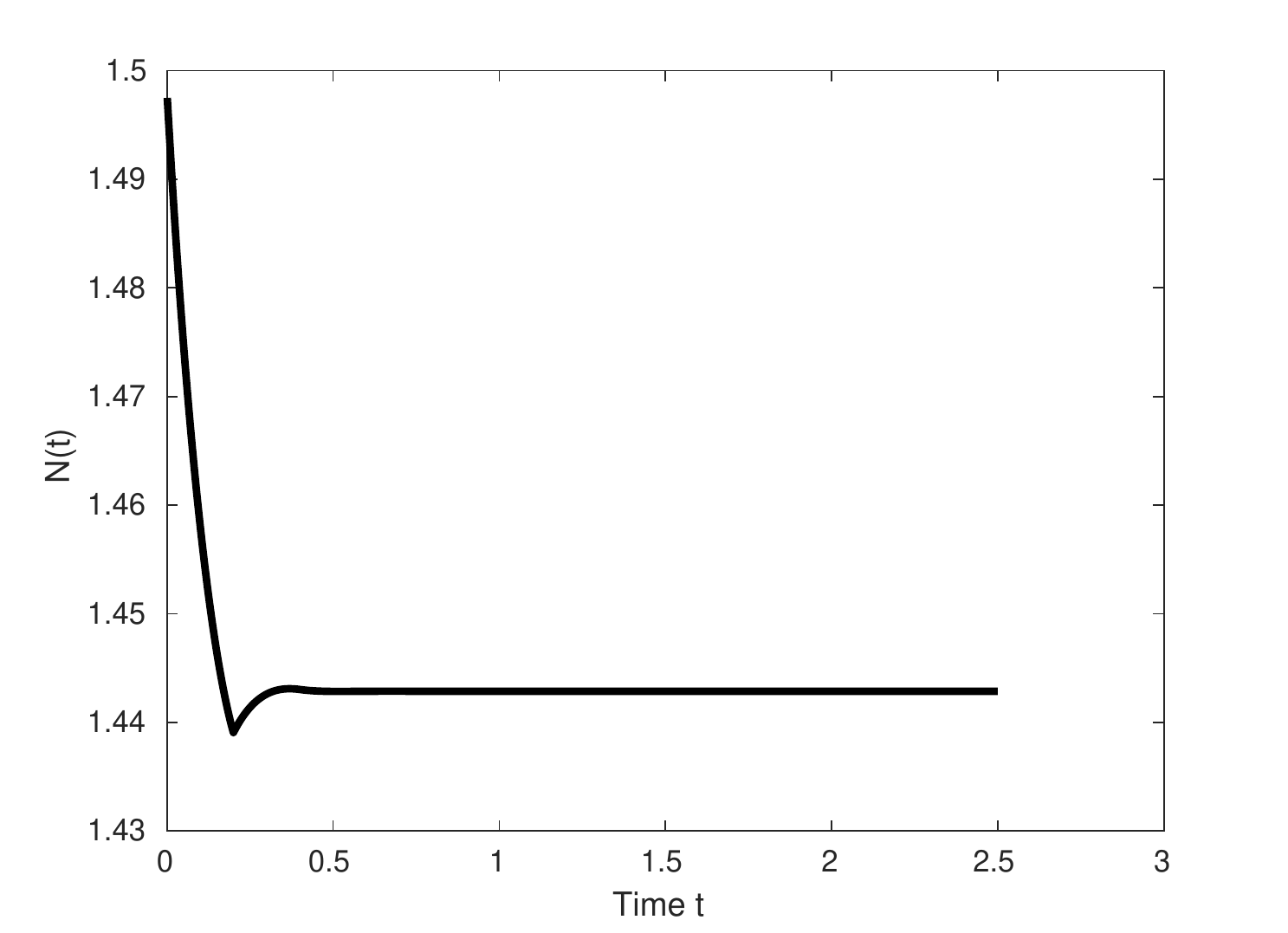}
		 \caption{Activity $N(t)$ for $N(0)=N_0^1$.}
		\label{CaseNExIn1}
	\end{subfigure}
	\begin{subfigure}{0.4\textwidth}
		\includegraphics[width=\textwidth]{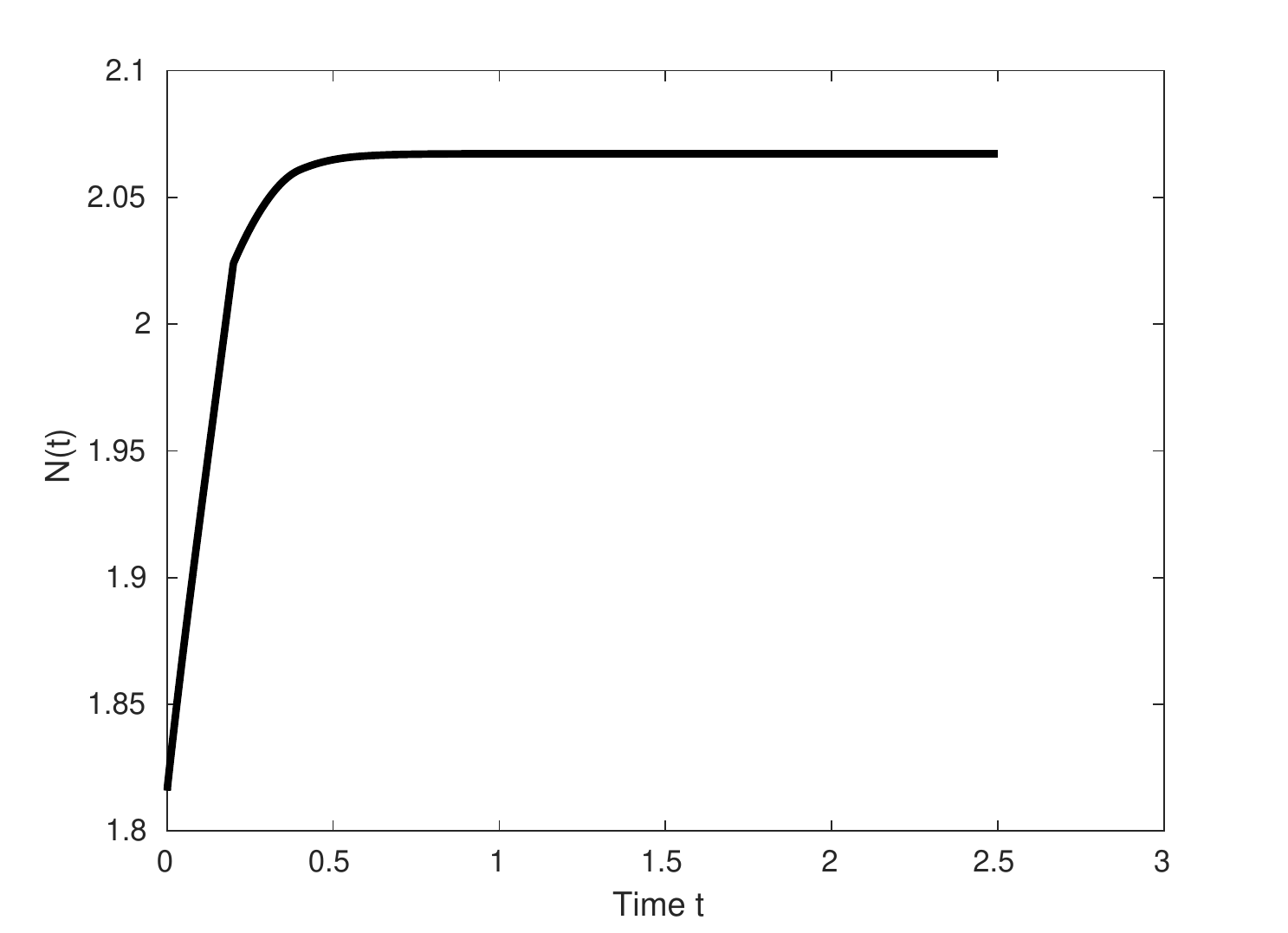}
		 \caption{Activity $N(t)$ for $N(0)=N_0^2$.}
		\label{CaseNExIn2}
	\end{subfigure}
	\begin{subfigure}{0.4\textwidth}
		\includegraphics[width=\textwidth]{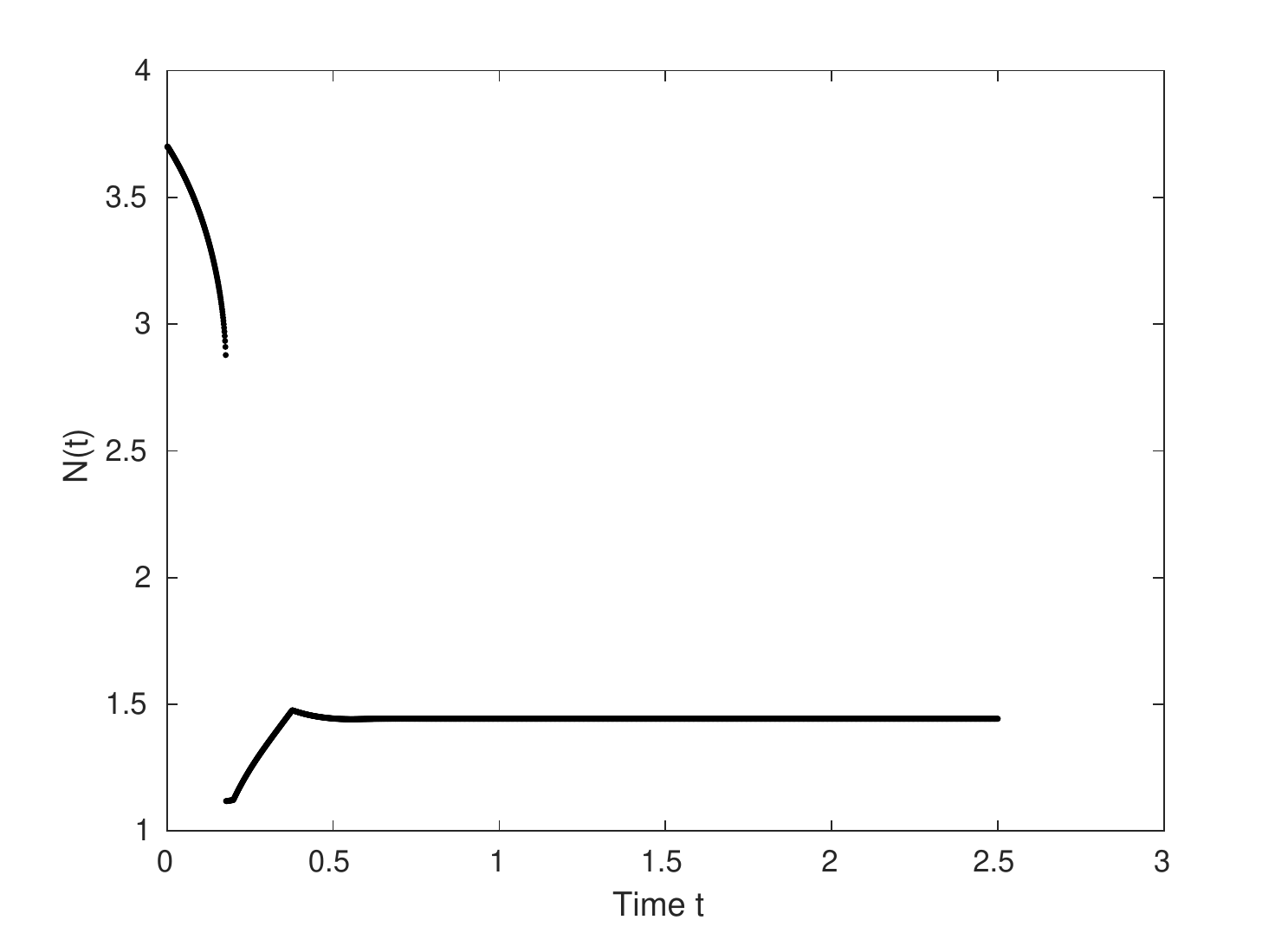}
		 \caption{Activity $N(t)$ for $N(0)=N_0^3$.}
		\label{CaseNExIn3}
	\end{subfigure}
	\begin{subfigure}{0.4\textwidth}
		\includegraphics[width=\textwidth]{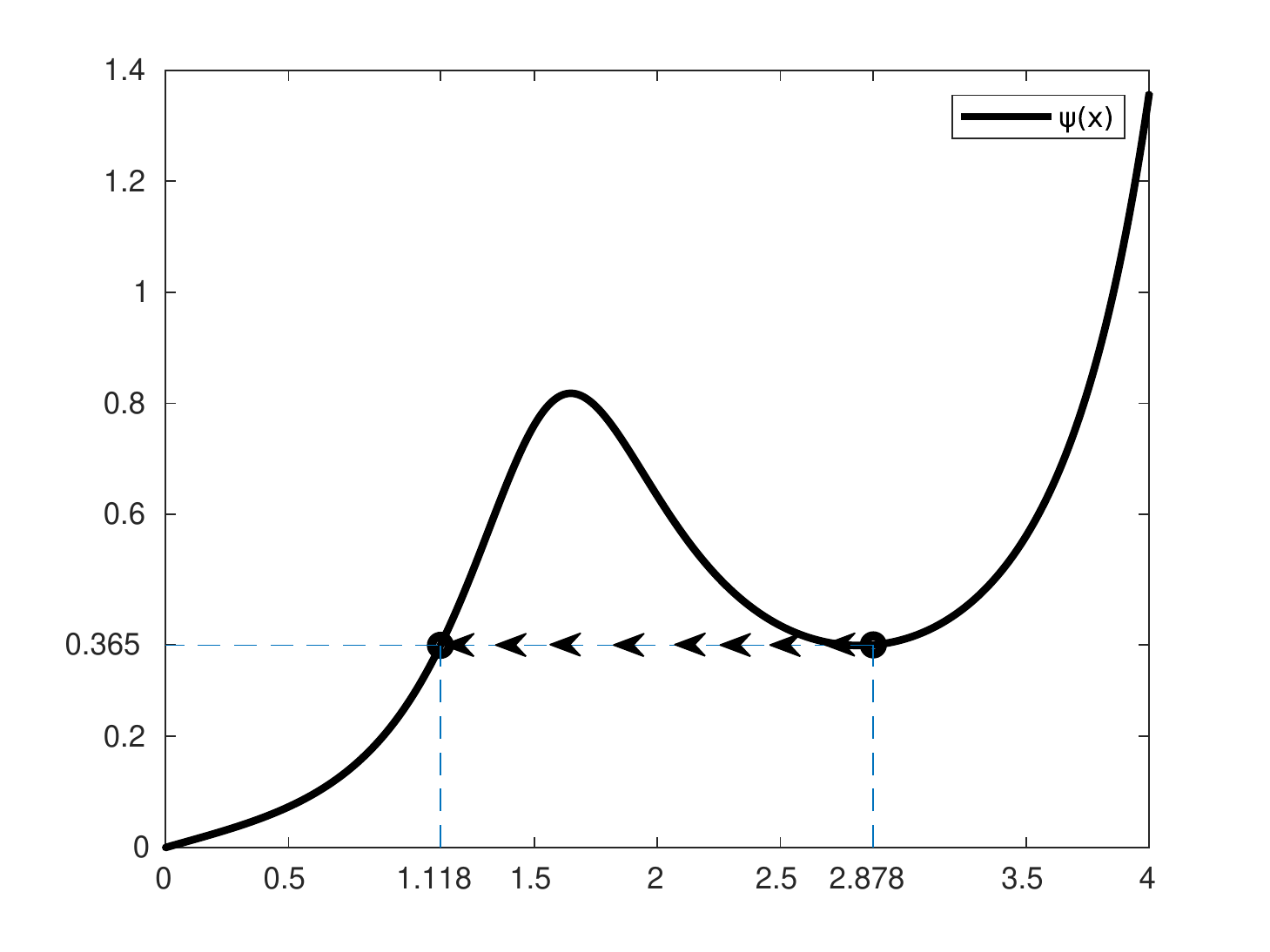}
		\caption{Jump discontinuity for the branch  of
                  $N(0)=N_0^3$ along the graph of $\psi$.}
        \label{SimplejumpExIn}
	\end{subfigure}
	  \\[-5pt]
        \caption{{\bf Example 4.}
          Activities of solutions with $n_0(s)=\frac{2}{3}(1+\cos(s))e^{-s}$
        for different values of $N(0)$.}
\end{figure}

In Figure \ref{CaseNExIn1} we observe that  $N(t)$ decreasing on $[0,\sigma]$ and the solution converges to the steady stated determined by $N_*^1$, which corresponds to the behaviour stated in Theorem \ref{convthm}. In Figure \ref{CaseNExIn2} we see that $N(t)$ converges monotonically to $N_*^2$, which $\psi'(N_*^2)<0$. In this case the solution is increasing and it corresponds to the behaviour stated in Theorem~\ref{monotoneconv}. Furthermore in Figure \ref{CaseNExIn3} we observe that $N(t)$ has a jump discontinuity at some $t_0\in(0,\sigma)$ that causes the solution to change to the branch of $N_0^1$ and then $N(t)$ converges to $N_*^1$ afterwards. At this jump time, the solution preserves the value of $\psi$ as we show in Figure \ref{SimplejumpExIn}. The horizontal arrows represent the change of $N(t)$ along the graph of $\psi$ at this discontinuity. We observe essentially the same behaviors of Example 2.

\section{Perspectives}
In the present  analysis of System \eqref{eq1} with the firing rate $p$ modulated by amplitude as given by \eqref{pfire}, we have exhibited some possible qualitative behaviours of solutions. Steady-state convergence always occurs in the  inhibitory and the weakly excitatory networks. And this can also occur for strongly excitatory connections, in particular situations. Periodic solutions can be built for strongly excitatory connections. Our method is based on the derivation of a nonlinear delay equation for the network activity. Moreover, numerical simulations are consistent with the theoretical results obtained about the convergence to equilibrium and the existence of jump discontinuities. 
 This study provides  possible behaviors which might arise for a more general firing rate. From this particular example of firing rate, we can think that the model induces an implicit delay, which is consistent with the interpretation of the discharge dynamics in the elapsed time model. 
 
Regarding convergence to equilibrium, we conjecture that the convergence rate in Theorem~\ref{convthm} for the inhibitory and weakly excitatory case is exponential, as it occurs in~\cite{PPD2} for a variant of the firing rate $p$, which is also given by an indicator function. Moreover, we expect the convergence in Theorem~\ref{monotoneconv} for monotone solutions in the strongly excitatory case to be exponential as well.

Concerning the existence of periodic solutions, it remains open to prove the existence of periodic continuous solutions when $p$ is not locally a linear function and also to find periodic solutions with a period other than a multiple of $\sigma$. With respect to stability of periodic orbits, an interesting question would be to determine what kind of piece-wise monotone solutions for the activity are stable in the excitatory case and if exponential convergence to this type of profile arises. We conjecture that the stable solutions correspond to those with few jump discontinuities in general.

\section*{Acknowledgements}
This project has received funding from the European Union's Horizon 2020 research and innovation program under the Marie Sklodowska-Curie grant agreement No 754362. It has also received support from ANR ChaMaNe No: ANR-19-CE40-0024.
\\[-35pt]
\begin{center}
    \includegraphics[scale=0.04]{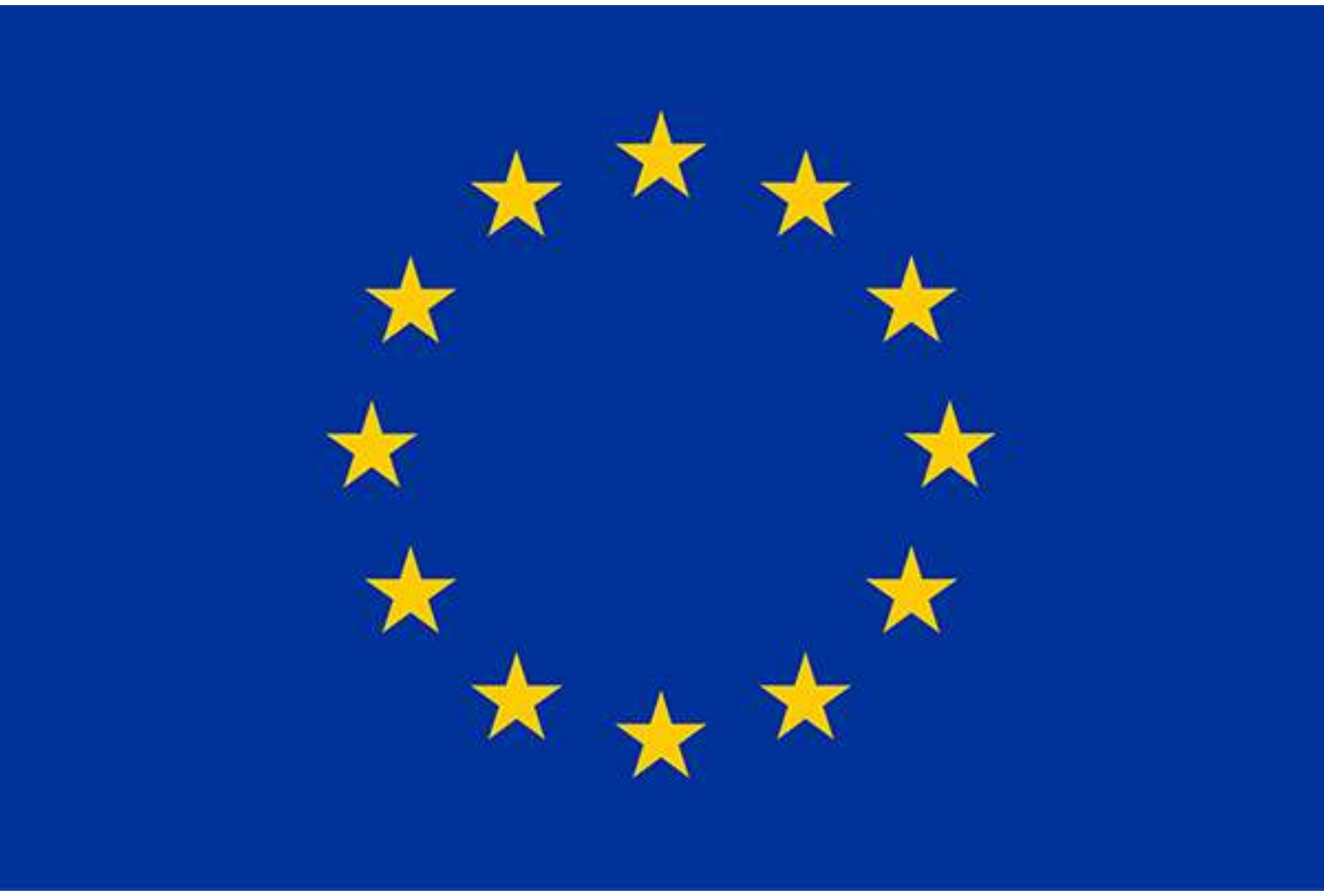}
\end{center}
\nocite{*}
\bibliography{biblio.bib}

\begin{thebibliography}{10}

\bibitem{brunel2000dynamics}
Nicolas Brunel.
\newblock Dynamics of sparsely connected networks of excitatory and inhibitory
  spiking neurons.
\newblock {\em Journal of computational neuroscience}, 8(3):183--208, 2000.

\bibitem{canizo2019asymptotic}
Jos{\'e}~A. Ca{\~n}izo and Havva Yolda{\c{s}}.
\newblock Asymptotic behaviour of neuron population models structured by
  elapsed-time.
\newblock {\em Nonlinearity}, 32(2):464, 2019.

\bibitem{chevallier2015mean}
Julien Chevallier.
\newblock Mean-field limit of generalized hawkes processes.
\newblock {\em Stochastic Processes and their Applications},
  127(12):3870--3912, 2017.

\bibitem{chevallier2015microscopic}
Julien Chevallier, Mar{\'\i}a~Jos{\'e} C{\'a}ceres, Marie Doumic, and Patricia
  Reynaud-Bouret.
\newblock Microscopic approach of a time elapsed neural model.
\newblock {\em Mathematical Models and Methods in Applied Sciences},
  25(14):2669--2719, 2015.

\bibitem{diekmann2012delay}
Odo Diekmann, Stephan~A. Van~Gils, Sjoerd M.~V. Lunel, and Hans-Otto Walther.
\newblock {\em Delay equations: functional-, complex-, and nonlinear analysis},
  volume 110.
\newblock Springer Science \& Business Media, 2012.

\bibitem{DH1}
G.~Dumont, J.~Henry, and C.~O. Tarniceriu.
\newblock Noisy threshold in neuronal models: connections with the noisy leaky
  integrate-and-fire model.
\newblock {\em J. Math. Biol.}, 73(6-7):1413--1436, 2016.

\bibitem{DH2}
Gr\'{e}gory Dumont, Jacques Henry, and Carmen~Oana Tarniceriu.
\newblock A theoretical connection between the noisy leaky integrate-and-fire
  and the escape rate models: the non-autonomous case.
\newblock {\em Math. Model. Nat. Phenom.}, 15:Paper No. 59, 20, 2020.

\bibitem{gerstner2002spiking}
Wulfram Gerstner and Werner~M. Kistler.
\newblock {\em Spiking neuron models: Single neurons, populations, plasticity}.
\newblock Cambridge university press, 2002.

\bibitem{hadeler1977periodic}
K.~P. Hadeler and J.~Tomiuk.
\newblock Periodic solutions of difference-differential equations.
\newblock {\em Arch. Rational Mech. Anal.}, 65(1):87--95, 1977.

\bibitem{kang2015}
Moon-Jin Kang, Beno{\^\i}t Perthame, and Delphine Salort.
\newblock Dynamics of time elapsed inhomogeneous neuron network model.
\newblock {\em Comptes Rendus Mathematique}, 353(12):1111--1115, 2015.

\bibitem{ly2009spike}
Cheng Ly and Daniel Tranchina.
\newblock Spike train statistics and dynamics with synaptic input from any
  renewal process: a population density approach.
\newblock {\em Neural Computation}, 21(2):360--396, 2009.

\bibitem{michel2005general}
Philippe Michel, St{\'e}phane Mischler, and Beno{\^\i}t Perthame.
\newblock General relative entropy inequality: an illustration on growth
  models.
\newblock {\em Journal de math{\'e}matiques pures et appliqu{\'e}es},
  84(9):1235--1260, 2005.

\bibitem{mischler2018}
St{\'e}phane Mischler and Qilong Weng.
\newblock Relaxation in time elapsed neuron network models in the weak
  connectivity regime.
\newblock {\em Acta Applicandae Mathematicae}, 157(1):45--74, 2018.

\bibitem{murray2007mathematical}
James~D. Murray.
\newblock {\em Mathematical biology: I. An introduction}, volume~17.
\newblock Springer Science \& Business Media, 2007.

\bibitem{PPD}
Khashayar Pakdaman, Beno{\^\i}t Perthame, and Delphine Salort.
\newblock Dynamics of a structured neuron population.
\newblock {\em Nonlinearity}, 23(1):55--75, 2010.

\bibitem{PPD2}
Khashayar Pakdaman, Beno{\^\i}t Perthame, and Delphine Salort.
\newblock Relaxation and self-sustained foscillations in the time elapsed
  neuron network model.
\newblock {\em SIAM J. Appl. Math.}, 73(3):1260--1279, 2013.

\bibitem{PPD3}
Khashayar Pakdaman, Beno{\^\i}t Perthame, and Delphine Salort.
\newblock Adaptation and fatigue model for neuron networks and large time
  asymptotics in a nonlinear fragmentation equation.
\newblock {\em J. Math. Neurosci.}, 4:Art. 14, 26, 2014.

\bibitem{perthame2006transport}
Beno{\^\i}t Perthame.
\newblock {\em Transport equations in biology}.
\newblock Springer Science \& Business Media, 2006.

\bibitem{pham1998activity}
Jo{\"e}l Pham, Khashayar Pakdaman, Jean Champagnat, and Jean-Fran{\c{c}}ois
  Vibert.
\newblock Activity in sparsely connected excitatory neural networks: effect of
  connectivity.
\newblock {\em Neural Networks}, 11(3):415--434, 1998.

\bibitem{smith2011introduction}
Hal~L. Smith.
\newblock {\em An introduction to delay differential equations with
  applications to the life sciences}, volume~57.
\newblock Springer New York, 2011.

\bibitem{ST2020}
Nicolas Torres and Delphine Salort.
\newblock Dynamics of neural networks with elapsed time model and learning
  processes.
\newblock {\em Acta Appl. Math.}, 170:1065--1099, 2020.

\end{thebibliography}
\bibliographystyle{plain}
\end{document}